\documentclass[10pt]{amsart}
\usepackage{amsthm,hyperref}
\usepackage{amsthm}
\usepackage{amssymb}
\usepackage[final]{showkeys}
\usepackage{microtype}
\usepackage{color}
\usepackage{graphicx}
\usepackage{adjustbox}
\usepackage[hmargin=3cm,vmargin={5cm,4cm}]{geometry}

\begin{document}
	
	\newtheorem{theorem}{Theorem}[section]
	\newtheorem{example}[theorem]{Example}
	\newtheorem{definition}[theorem]{Definition}
	\newtheorem{remark}[theorem]{Remark}
	\newtheorem{proposition}[theorem]{Proposition}
	\newtheorem{lemma}[theorem]{Lemma}
	\newtheorem{corollary}[theorem]{Corollary}
	\theoremstyle{definition}
	\newtheorem{conjecture}[theorem]{Conjecture}
	\newtheorem{problem}[theorem]{Problem}
	\newtheorem{question}[theorem]{Question}
	
	\numberwithin{equation}{section}
	
	\renewcommand{\subjclassname}{%
		\textup{2020} Mathematics Subject Classification}
	
	\allowdisplaybreaks
	
	\title[]{The nested derivative criterion for determining elementary solutions to certain non-linear PDEs  in one-dimensional space-time.}	
	\author{Francesco Maltese}
	\address{Liceo Scientifico Amedeo Avogadro}
	\email{ennio.vincent@gmail.com}
	

	\maketitle
	
	\begin{abstract}
		
	In this article, we have studied a particular criterion to establish if a particular class of nonlinear PDEs $\mathcal{O}_t u=\mathcal{O}_x u $ in unidimensional space-time $(x,t)$ admits  elementary solutions respect to the spatial coordinate $x$, through the nested derivative method. Where a nested derivative is a generalization of the usual partial derivative of a order higher than one and, in terms of differentials operators, is a generalization of the Laguerre operator. This method, integrated with  Liouville's theorem and Galois theory and method by integrate by parts, establishes if solution of nested derivative equations admits elementary solutions respect with $x$. These solutions could help to study better evolutions of possible physical systems in space-time.

	\bigskip

	\quad

	\end{abstract}

	\section{Introduction}\label{Se:Section_1}
	
	In this section we will use certain criteria to establish whether there are the elementary solutions and then determine them for non-linear PDE, which could represent a hypothetical physical model in unidimensional space-time $(x,t)$ of the type
	
	$$\mathcal{O}_t u=\mathcal{O}_x u .$$

Where  $\mathcal{O}_x$, $\mathcal{O}_t$ are differentials operators
 with respect to the spatial variable and the temporal variable, respectively, and $u$ is a function in unidimensional space-time $(x,t)$. In particular, the operators  $\mathcal{O}_x$ will be polynomial with respect to $u$ and its partial derivatives with respect to $x$ i.e 
	 $	\mathcal{O}_x u \in \mathcal{C}^{\infty}(U;\mathbb{C})\big[\partial^{(k)}_{x}u \big]_{k=0}^{l}$.
	 \quad
	 where $U$ is an opportune open set on space-time $(x,t)$ and we recall that 
	
	$$\big[\partial^{(k)}_{x}u \big]_{k=0}^{l}=\big[\partial^{(0)}_{x}u=u,\partial^{(1)}_{x}u,...,\partial^{(k)}_{x}u,...,\partial^{(l)}_{x}u \big ]$$ and $\mathcal{C}^{\infty}(U;\mathbb{C})\big[\partial^{(k)}_{x}u \big]_{k=0}^{l}$ can be seen as ring of polynomials in indeterminates $\partial^{(k)}_{x}u$ with coefficients in $\mathcal{C}^{\infty}(U;\mathbb{C})$.
	
	\quad
	
	We will solve these equations in two ways:
	
	$$ 1)\mathcal{O}_t u=0 , \mathcal{O}_x u=0     $$
	$$ 2)\mathcal{O}_x u \in W=<g_1(x),...,g_s(x)>, \hspace{0.1cm} if\hspace{0.1cm} u(x,t) \in W \hspace{0.1cm}\mathcal{O}_t h_i(t)=\lambda(t) h_i(t)\hspace{0.1cm} 1\leq i \leq s,$$
	
	if $u(x,t)= \sum_{i=1}^{s}h_i(t)g_i(x)$.

	\quad

	The second way is called "the method invariant space"(for details of this method, please refer to \cite{gazizov2013construction} and/or \cite[Section 5]{maltese2024non}) with an opportune function $\lambda(t)$ . In this case we will use, as operator $\mathcal{O}_t$, a linear operator. In the case $s=1$ we are considering the solution $u$ as an eigenfunction of operator $\mathcal{O}_x$ with an opportune eigenvalue $\lambda(t)$.
	
	\quad

	 In both cases 1) and 2), we can see  $\mathcal{O}_x u$ as a sum of two operators in this way:
	
	\begin{equation}\label{eq:Equation_1}
		 	\mathcal{O}_x u=\mathcal{\overline{O}}_x u + \lambda(t) u ,
	\end{equation}
	where, if there is an eigenfunction $\bar{u} \in \mathcal{C}^{\infty}(U;\mathbb{C})$ with eigenvalue $\lambda(t)$ for $	\mathcal{O}_x u$, $\mathcal{\overline{O}}_x \bar{u}=0$, in particular in the way 1) we will have $\lambda(t) \equiv 0$. In this case, furthermore, we can also remark that $\mathcal{\overline{O}}_x u \in \big(\partial^{(k)}_{x}u-\partial^{(k)}_{x}\bar{u}\big)_{k=0}^{l} $  an ideal of $\mathcal{C}^{\infty}(U;\mathbb{C})\big[\partial^{(k)}_{x}u \big]_{k=0}^{l}$ generated of operators $\partial^{(k)}_{x}u-\partial^{(k)}_{x}\bar{u}$ and by \eqref{eq:Equation_1} we get that $\mathcal{O}_x u \in \lambda(t) u + \big(\partial^{(k)}_{x}u-\partial^{(k)}_{x}\bar{u}\big)_{k=0}^{l} $ that is $\mathcal{O}_x u$ belongs to the class lateral $\lambda(t) u$ modulo $ \big(\partial^{(k)}_{x}u-\partial^{(k)}_{x}\bar{u}\big)_{k=0}^{l} $ . Then we can assert that
	
	\begin{equation}\label{pr:prop_1}
		\mathcal{O}_x \bar{u}= \lambda \bar{u} \hspace{0.1cm} \iff \hspace{0.1cm}  \mathcal{O}_x u \in \lambda(t) u + \big(\partial^{(k)}_{x}u-\partial^{(k)}_{x}\bar{u}\big)_{k=0}^{l}.
	\end{equation}
	
	\quad

	Therefore, in this article we will focus on the operator $\mathcal{\overline{O}}_x u$ of \eqref{eq:Equation_1} and attempt to solve the EDO
	
\begin{equation}\label{eq:Equation_2}
	\mathcal{\overline{O}}_x u=0,
\end{equation}

	  using the method of nested derivatives(for the definition of nested derivative, see \eqref{eq:Equation_4}) method where possible. In fact, we will find the conditions under which the operator $\mathcal{\overline{O}}_x u$ can be expressed with nested derivatives, and once the solution has been found using this method, we will establish when a solution $\bar{u}$ of \eqref{eq:Equation_2} is elementary.
	  
	  \quad
	  
	  The nested derivative can be viewed as a generalization of the Laguerre operator $\partial_t t \partial_t$(see \cite{ricci2020laguerre}). In most of the examples given in this article, it will be applied with respect to the spatial coordinate. The structure of the nested derivative operator allows us to understand the building blocks of the solution, and from this we can determine whether some solutions are elementary.

	     \section{Notations}\label{se:Section_2}
	     
	     Now, so as not to further weigh down the development of 	$\mathcal{\overline{O}}_x u$, we can view it in this way 
	     
	     \begin{equation}\label{eq:Equation_3}
	     	\mathcal{\overline{O}}_x u=\sum_{I \in \mathcal{F}} A_I (x,t) \mathcal{P}_{l,x}^{I}u,
	     \end{equation} 
	
	where $l>0$ is a fixed positive integer and $\mathcal{F}$ is a finite sequence of points in $\mathbb{N}^{l+1}$ and $\mathcal{P}_{l,x}^{I}$ is a non-linear algebraic differential operator defined in this way
	
	\begin{equation}\label{eq:Equation_4}
	\mathcal{P}_{l,x}^{I}u:=\prod_{j=0}^{l}\big(\partial_{x}^{(j)}u\big)^{i_j},
	\end{equation}
	
	if $I=(i_0, i_1,..., i_j,...,i_l )$.
	\quad

	We will use the following conventions on the multi-index $I \in \mathcal{F}$; if we have, for example, $I,J,K \in \mathcal{F}$ their expressions by components will be named in their indexed small versions, i.e., $I=(i_0, i_1,..., i_j,...,i_l )$, $J=(j_0, j_1,..., j_j,...,j_l )$, $K=(k_0, k_1,..., k_j,...,k_l )$.
	
	\quad
	
	Furthermore, we will suppose that $l$ of \eqref{eq:Equation_3} and \eqref{eq:Equation_4} is the highest degree of derivation of $	\mathcal{\overline{O}}_x u$, and since $	\mathcal{\overline{O}}_x u$ is non-linear, then it will be at least $I \in \mathcal{F}$ such that $i_j >1$ for some $j$.
	
	\section{The nested derivative}\label{se:Section_3}
	
	\begin{definition}\label{df:Definition_1}
	 A nested derivative of order $k>1$ is the following differential operator.
	 
	 \begin{equation}\label{eq:Equation_5}
	 	\partial_{x}f_{k-1}\partial_{x}f_{k-2}\partial_{x}...\partial_{x}f_{l}\partial_{x}...\partial_{x}f_{2}\partial_{x}f_{1}\partial_{x}.
	 \end{equation}
	 	with $f_{l}$ for $ 1\leq l \leq k-1$ functions of class at least $C^{k}$ 
		on a suitable open interval of $\mathbb{R}.$
	\end{definition}

\quad

\begin{remark}\label{os:Observation_1}
We can see the usual derivation of order $k$, $\partial_{x}^{(k)}$	as a special case of the nested derivative for $f_{k-1}=f_{k-2}=....=f_{l}=....=f_{1}=1$	for $k>1$.
\end{remark}

\quad

Furthermore, to compact the notation, we will denote with $D_{f_{l}}:=f_{l}\partial_{x}$. In this way \eqref{eq:Equation_5} becomes

 \begin{equation}\label{eq:Equation_6}
\partial_{x}D_{f_{l},1}^{(k-1)}:=	\partial_{x}D_{f_{k-1}}D_{f_{k-2}}...D_{f_{l}}...D_{f_{2}}D_{f_{1}}.
\end{equation}

If in \eqref{eq:Equation_6} we get that $f_{k-1}=f_{k-2}=....=f_{l}=....=f_{1}=f$, the nested derivative of order $k$ can be denoted, in a natural manner, as $\partial_{x}D_{f}^{(k-1)}$, where $D_{f}^{(k-1)}=\underbrace{D_{f}D_{f}...D_{f}...D_{f}D_{f}}_{(k-1)-times}$ and as the usual derivation, we will denote $D_{f}^{(0)}u(x):=u(x)$.

Now, given a function $f(x)$ of class at least $C^{(k)}$, we want to explicity compute $\partial_{x}D_{f_{l},1}^{(k-1)}f$.
Let's begin with case $k=2$ and $k=3$.

\quad

$$\partial_{x}D_{f_{l},1}^{(1)}f=\partial_{x}f_{1}\partial_{x}f=f_{1}^{'}\partial_{x}f+f_{1}\partial_{x}^{(2)}f=n_{1}^{(2)}f_{1}^{'}\partial_{x}f+n_{0}^{(2)}f_{1}\partial_{x}^{(2)}f,$$

\quad
with $n_{1}^{(2)}=1$, $n_{0}^{(2)}=1$.

\quad

For $k=3$, we get 

$$\partial_{x}D_{f_{l},1}^{(2)}f=\partial_{x}D_{f_{2}}D_{f_{1}}f=\partial_{x}f_{2}\partial_{x}f_{1}\partial_{x}f$$

$$=\partial_{x}f_{2}(n_{1}^{(2)}f_{1}^{'}\partial_{x}f+n_{0}^{(2)}f_{1}\partial_{x}^{(2)}f)$$ $$=\partial_{x}(n_{1}^{(2)}f_{1}^{'}f_{2}\partial_{x}f+n_{0}^{(2)}f_{1}f_{2}\partial_{x}^{(2)}f)$$

\begin{equation}\label{eq:Equation_7}
=n_{1}^{(2)}f_{1}^{(2)}f_{2}\partial_{x}f+n_{1}^{(2)}f_{1}^{'}f_{2}^{'}\partial_{x}f+n_{1}^{(2)}f_{1}^{'}f_{2}\partial_{x}^{(2)}f
\end{equation}

$$+n_{0}^{(2)}f_{1}^{'}f_{2}\partial_{x}^{(2)}f+n_{0}^{(2)}f_{1}f_{2}^{'}\partial_{x}^{(2)}f+n_{0}^{(2)}f_{1}f_{2}\partial_{x}^{(3)}f.$$

 Now we make a partial gathering on \eqref{eq:Equation_7} with respect  $\partial_{x}^{(j)}f$ and obtain
 
 $$( n_{1}^{(2)}f_{1}^{(2)}f_{2}+n_{1}^{(2)}f_{1}^{'}f_{2}^{'})\partial_{x}f+[(n_{1}^{(2)}+n_{0}^{(2)})f_{1}^{'}f_{2}+n_{0}^{(2)}f_{1}f_{2}^{'}]\partial_{x}^{(2)}f+n_{0}^{(2)}f_{1}f_{2}\partial_{x}^{(3)}f=$$
 
 \begin{equation}\label{eq:Equation_8}
 (n_{2,0}^{(3)}f_{1}^{(2)}f_{2}+n_{1,1}^{(3)}f_{1}^{'}f_{2}^{'})\partial_{x}f+(n_{1,0}^{(3)}f_{1}^{'}f_{2}+n_{0,1}^{(3)}f_{1}f_{2}^{'}) \partial_{x}^{(2)}f +n_{0,0}^{(3)}f_{1}f_{2}\partial_{x}^{(3)}f,
\end{equation}

\quad 
with 

\begin{equation}\label{eqs:Equations_1}
	n_{2,0}^{(3)}=n_{1}^{(2)}, \vspace{0.3cm}n_{1,1}^{(3)}=n_{1}^{(2)},\vspace{0.3cm}n_{1,0}^{(3)}=n_{1}^{(2)}+n_{0}^{(2)},\vspace{0.3cm}n_{0,1}^{(3)}=n_{0}^{(2)}, \vspace{0.3cm}n_{0,0}^{(3)}=n_{0}^{(2)}. 
\end{equation}

We can find a general formula to calculate $\partial_{x}D_{f_{l},1}^{(k-1)}f$ for $k>3$	
based on what we determined from k=2 to k=3. We denote with $n_{i_1,...,i_l,...,i_{j-1}}^{(j)}$ the numerical coefficient of the term $ \prod_{l=1}^{j-1}f_{l}^{(i_l)}$ generated from the development of $\partial_{x}D_{f_{l},1}^{(j-1)}f$.

\quad
The subscripts $i_1,...,i_l,...,i_{j-1}$ indicate the order of differentiation of the functions $f_l$; therefore, given the definition of the nested derivative of order $j$, these subscripts are subject to restrictions depending on the order in which they appear in the expansion of the nested derivative, i.e., $i_l \leq j-l$. If we start with the constraints of subscript to determine $\partial_{x}D_{f_{l},1}^{(j-1)}f$, it could happen that some subscripts $i_1,...,i_l,...,i_{j-1}$ are compatible with constraints, as just explained above, but such subscripts don't occur in the calculus of $\partial_{x}D_{f_{l},1}^{(j-1)}f$. 

\quad

First, let us look at this and why it is so, using an example; we will then demonstrate it in general terms. The example involves applying the definition \ref{df:Definition_1} in the case where $k=4$; this is rather laborious but necessary in order to provide a rigorous, unambiguous definition of the terms $n_{i_1,...,i_l,...,i_{j-1}}^{(j)}$.

\quad

For $k=4$ we get 

$$\partial_{x}D_{f_{l},1}^{(3)}=\partial_{x}f_3 [(n_{2,0}^{(3)}f_{1}^{(2)}f_{2}+n_{1,1}^{(3)}f_{1}^{'}f_{2}^{'})\partial_{x}f+(n_{1,0}^{(3)}f_{1}^{'}f_{2}+n_{0,1}^{(3)}f_{1}f_{2}^{'}) \partial_{x}^{(2)}f +n_{0,0}^{(3)}f_{1}f_{2}\partial_{x}^{(3)}f]=$$

$$ = \partial_{x}[(n_{2,0}^{(3)}f_{1}^{(2)}f_{2}f_3 +n_{1,1}^{(3)}f_{1}^{'}f_{2}^{'}f_3)\partial_{x}f+(n_{1,0}^{(3)}f_{1}^{'}f_{2}f_3 +n_{0,1}^{(3)}f_{1}f_{2}^{'}f_3) \partial_{x}^{(2)}f +n_{0,0}^{(3)}f_{1}f_{2}f_3\partial_{x}^{(3)}f ] = $$

$$=n_{1,1}^{(3)}f'_3f'_2f'_1\partial_{x}f+n_{1,1}^{(3)}f_3f''_2 f'_1\partial_x f+ n_{1,1}^{(3)}f_3f'_2 f''_1\partial_x f +n_{1,1}^{(3)}f_3f'_2 f'_1\partial_{x}^{(2)}  f+n_{2,0}^{(3)}f'_3f_2f''_1\partial_{x}f +$$

$$n_{2,0}^{(3)}f_3f'_2f''_1\partial_{x}f+n_{2,0}^{(3)}f_3f_2f'''_1\partial_{x}f+n_{2,0}^{(3)}f_3f_2f''_1\partial_{x}^{(2)}f+n_{1,0}^{(3)}f_{3}^{'}f_{2}f'_1\partial_{x}^{(2)}f+ $$

$$n_{1,0}^{(3)}f_{3}f'_{2}f'_1\partial_{x}^{(2)}f+ n_{1,0}^{(3)}f_{3}f_{2}f''_1\partial_{x}^{(2)}f+n_{1,0}^{(3)}f_{3}f_{2}f'_1\partial_{x}^{(3)}f + n_{0,1}^{(3)}f'_{3}f'_{2}f_1\partial_{x}^{(2)}f  + n_{0,1}^{(3)}f_{3}f''_{2}f_1\partial_{x}^{(2)}f+$$
$$n_{0,1}^{(3)}f_{3}f'_{2}f'_1\partial_{x}^{(2)}f+n_{0,1}^{(3)}f_{3}f'_{2}f_1\partial_{x}^{(3)}f+n_{0,0}^{(3)}f'_{3}f_{2}f_1\partial_{x}^{(3)}f +n_{0,0}^{(3)}f_{3}f'_{2}f_1\partial_{x}^{(3)}f +n_{0,0}^{(3)}f_{3}f_{2}f'_1\partial_{x}^{(3)}f+$$

$$n_{0,0}^{(3)}f_{3}f_{2}f_1\partial_{x}^{(4)}f.$$

$$\partial_{x}D_{f_{l},1}^{(3)}=[ n_{1,1}^{(3)}f'_3f'_2f'_1+n_{1,1}^{(3)}f_3f''_2 f'_1\partial_x f+ (n_{1,1}^{(3)}+ n_{2,0}^{(3)})f_3f'_2 f''_1 +n_{2,0}^{(3)} f'_3f_2f''_1 +n_{2,0}^{(3)}f_3f_2f'''_1  ]\partial_{x}f+$$
$$[(n_{1,1}^{(3)}+n_{1,0}^{(3)}+n_{0,1}^{(3)})f_3 f'_2f'_1+(n_{2,0}^{(3)}+n_{1,0}^{(3)})f_3 f_2 f''_1+ n_{1,0}^{(3)}f'_3 f_2 f'_1+n_{0,1}^{(3)}f'_3 f'_2 f_1+n_{0,1}^{(3)}f_3 f''_2 f_1 ]\partial_{x}^{(2)}f+ $$

$$[(n_{1,0}^{(3)}+n_{0,0}^{(3)})f_3 f_2 f'_1+ (n_{0,1}^{(3)}+n_{0,0}^{(3)})f_3 f'_2 f_1+n_{0,0}^{(3)}f'_3f_2f_1]\partial_{x}^{(3)}f+n_{0,0}^{(3)}f_3 f_2 f_1 \partial_{x}^{(4)}f. $$

\quad

By \eqref{eqs:Equations_1}, we can set up and deduce the following equalities

\begin{equation}\label{eqs:Equations_3}
n_{1,1,1}^{(4)}=n_{1,1}^{(3)}=1, \hspace{0.1cm}	n_{1,2,0}^{(4)}=n_{1,1}^{(3)}=1,\hspace{0.1cm}n_{2,1,0}^{(4)}=n_{2,0}^{(3)}+n_{1,1}^{(3)}=2,\hspace{0.1cm}n_{2,0,1}^{(4)}=n_{2,0}^{(3)}=1,\hspace{0.1cm}n_{3,0,0}^{(4)}=n_{2,0}^{(3)}=1.
\end{equation}
$$n_{1,1,0}^{(4)}=n_{1,1}^{(3)}+n_{1,0}^{(3)}+n_{0,1}^{(3)}=4, n_{2,0,0}^{(4)}=n_{2,0}^{(3)}+n_{1,0}^{(3)}=3,\hspace{0.1cm}n_{1,0,1}^{(4)}=n_{1,0}^{(3)}=2,\hspace{0.1cm} n_{0,1,1}^{(4)}= n_{0,1}^{(3)}=1. $$
$$n_{0,2,0}^{(4)}= n_{0,1}^{(3)}=1, \hspace{0.1cm}n_{1,0,0}^{(4)}=n_{1,0}^{(3)}+n_{0,0}^{(3)}=3,\hspace{0.1cm}n_{0,1,0}^{(4)}=n_{0,1}^{(3)}+n_{0,0}^{(3)}=2,\hspace{0.1cm} n_{0,0,1}^{(4)}=n_{0,0}^{(3)}=1, \hspace{0.1cm}n_{0,0,0}^{(4)}=n_{0,0}^{(3)}=1.$$

\quad
The equalities \eqref{eqs:Equations_3} complete the calculation of $\partial_{x}D_{f_{l},1}^{(3)}$.
\begin{equation}\label{eq:Equation_19}
	\partial_{x}D_{f_{l},1}^{(3)}=\big( n_{1,1,1}^{(4)}f'_3f'_2f'_1+ n_{1,2,0}^{(4)}f_3f''_2f'_1+n_{2,1,0}^{(4)}f_3f'_2f''_1+n_{2,0,1}^{(4)}f'_3f_2f''_1+n_{3,0,0}^{(4)}f_3f_2f'''_1 \big)\partial_x f+
\end{equation}	$$\big(n_{1,1,0}^{(4)}f_3f'_2f'_1+n_{2,0,0}^{(4)}f_3f_2f''_1 + n_{1,0,1}^{(4)}f'_3f_2f'_1+n_{0,1,1}^{(4)}f'_3f'_2f_1+ n_{0,2,0}^{(4)}f_3f''_2f_1 \big)\partial_{x}^{(2)}f+$$
$$\big(n_{1,0,0}^{(4)}f_3f_2f'_1+ n_{0,1,0}^{(4)}f_3f'_2f_1+n_{0,0,1}^{(4)}f'_3f_2f_1 \big)\partial_{x}^{(3)}f+n_{0,0,0}^{(4)}f_3f_2f_1\partial_{x}^{(4)}f.$$

\quad

From \eqref{eqs:Equations_3}, we can draw the following conclusion: the subscripts $i_1,i_2,i_3$ which describe the degrees of derivation of $f_3^{(i_3)}f_2^{(i_2)}f_1^{(i_1)}$ where for the case $k=4$ we have $i_3\leq 1, \hspace{0.1cm} i_2\leq 2,\hspace{0.1cm} i_1 \leq 3$. So the constraints are met, but despite that in \eqref{eqs:Equations_3} one can see that only one $i_j$ in the subscripts $i_1,i_2,i_3$ can reach the maximum value of its constraints in $n_{i_1,i_2,i_3}^{(4)}$. This is due to the fact that in order to calculate $	\partial_{x}D_{f_{l},1}^{(3)}$, it multiplies by $f_3$ and then it derives the terms $f_3f_2^{(\tilde{i}_2)}f_3^{(\tilde{i}_3)}$ by a single factor $f_j^{(\tilde{i}_j)}$.
Therefore, to generalize this fact, let us define the following sets

$$\mathcal{I}max\big((i_l)_{l=1}^{j-1}\big)=\mathcal{I}max(i_1,...,i_l,...,i_{j-1}):=\{ 1\leq l \leq j-1 | i_l=j-l \}.$$

\quad
So, to summarize what has just been said, in \eqref{eq:Equation_19} it's impossible  for the coefficients $n_{i_1,i_2,i_3}^{(4)}$ to get $|\mathcal{I}max\big((i_l)_{l=1}^{3}\big)|>1$.

 Now we have the tools to prove that this in general and the following

\begin{proposition}\label{pr:Proposition_1}
	For $k>2$
	\begin{equation}\label{eq:Equation_9}
		\partial_{x}D_{f_{l},1}^{(k-1)}f=\sum_{j=1}^{k}\bigg(\sum_{\substack
			{\sum_{l=1}^{k-1}i_{l}=k-j,
		       0\leq i_l \leq k-l }} n_{i_1,...,i_l,...,i_{k-1}}^{(k)}\prod_{l=1}^{k-1}f_{l}^{(i_l)}\bigg)\partial_{x}^{(j)}f.
	\end{equation}
\quad
Where the coefficients $n_{i_1,...,i_l,...,i_{k-1}}^{(k)}$ enjoy the following recursive property

\begin{equation}\label{eqs:Equations_2}
n_{i_1,...,i_l,...,i_{k-1}}^{(k)}=\begin{cases} 0\hspace{0.1cm} If\hspace{0.1cm} |\mathcal{I}max\big((i_l)_{l=1}^{k-1}\big)|>1 \\
n_{i_1,...,i_l,...,i_{k-2}}^{(k-1)}\hspace{0.1cm}If\hspace{0.1cm} |\mathcal{I}max\big((i_l)_{l=1}^{k-1}\big)|=1\hspace{0.1cm} and \hspace{0.1cm} \mathcal{I}max\big((i_l)_{l=1}^{k-1}\big)=\{ k-1\} \\
n_{i_1,...,i_l-1,...,i_{k-2}}^{(k-1)}\hspace{0.1cm}If\hspace{0.1cm}|\mathcal{I}max\big((i_l)_{l=1}^{k-1}\big)|=1\hspace{0.1cm} and \hspace{0.1cm}\mathcal{I}max\big((i_l)_{l=1}^{k-1}\big)=\{ l\}  \hspace{0.1cm}with\hspace{0.1cm} l \neq k-1\\
n_{i_1,...,i_l,...,i_{k-2}}^{(k-1)}+\sum_{l=1, i_l\neq 0}^{k-2}n_{i_1,...,i_l-1,...,i_{k-2}}^{(k-1)}\hspace{0.1cm} Otherwise \hspace{0.1cm} if \hspace{0.1cm} |\mathcal{I}max\big((i_l)_{l=1}^{k-1}\big)|=0,
\end{cases}
\end{equation}
such that $ 0\leq i_l \leq k-l $ and $\sum_{l=1}^{k-1}i_{l}=k-j $ with $n_{1}^{(2)}=1$, $n_{0}^{(2)}=1$.

\end{proposition}
	\begin{proof}
We prove \eqref{eq:Equation_9}	and \eqref{eqs:Equations_2} by induction on $k$. We choose, as the base of induction $k=4$, the equations \eqref{eq:Equation_19}, as we discussed earlier, and \eqref{eqs:Equations_3} proves \ref{pr:Proposition_1}. Let's move on to the inductive hypothesis .
We know that
$$	\partial_{x}D_{f_{l},1}^{(k-2)}f=\sum_{j=1}^{k-1}\bigg(\sum_{\substack
	{\sum_{l=1}^{k-2}i_{l}=k-1-j,
		0\leq i_l \leq k-1-l }} n_{i_1,...,i_l,...,i_{k-2}}^{(k-1)}\prod_{l=1}^{k-2}f_{l}^{(i_l)}\bigg)\partial_{x}^{(j)}f   .$$
	\quad
	So $$	\partial_{x}D_{f_{l},1}^{(k-1)}f= \partial_{x}f_{k-1}\partial_{x}D_{f_{l},1}^{(k-2)}f$$
	\quad
	and applying the derivation of the product, we get
	
	$$ \partial_{x}D_{f_{l},1}^{(k-1)}f=\partial_{x}f_{k-1}\bigg[\sum_{j=1}^{k-1}\bigg(\sum_{\substack
		{\sum_{l=1}^{k-2}i_{l}=k-1-j,
			0\leq i_l \leq k-1-l }} n_{i_1,...,i_l,...,i_{k-2}}^{(k-1)}\prod_{l=1}^{k-2}f_{l}^{(i_l)}\bigg)\partial_{x}^{(j)}f    \bigg]=$$
		
		$$ = f_{k-1}^{(1)}\sum_{j=1}^{k-1}\bigg(\sum_{\substack
			{\sum_{l=1}^{k-2}i_{l}=k-1-j,
				0\leq i_l \leq k-1-l }} n_{i_1,...,i_l,...,i_{k-2}}^{(k-1)}\prod_{l=1}^{k-2}f_{l}^{(i_l)}\bigg)\partial_{x}^{(j)}f$$    $$+f_{k-1}\sum_{j=1}^{k-1}\bigg(\sum_{\substack
			{\sum_{l=1}^{k-2}i_{l}=k-1-j,
			0\leq i_l \leq k-1-l }} n_{i_1,...,i_l,...,i_{k-2}}^{(k-1)}\partial_{x}\prod_{l=1}^{k-2}f_{l}^{(i_l)}\bigg)\partial_{x}^{(j)}f$$
		
		$$ +  f_{k-1}\sum_{j=1}^{k-1}\bigg(\sum_{\substack
			{\sum_{l=1}^{k-2}i_{l}=k-1-j,
				0\leq i_l \leq k-1-l }} n_{i_1,...,i_l,...,i_{k-2}}^{(k-1)}\prod_{l=1}^{k-2}f_{l}^{(i_l)}\bigg)\partial_{x}^{(j+1)}f  =         $$

			\begin{equation}\label{eq:Equation_10}
			=\bigg(\sum_{\substack
				{\sum_{l=1}^{k-2}i_{l}=k-2,
					0\leq i_l \leq k-1-l }} n_{i_1,...,i_l,...,i_{k-2}}^{(k-1)}f_{k-1}^{(1)}\prod_{l=1}^{k-2}f_{l}^{(i_l)}\bigg)\partial_{x}^{(1)}f
					\end{equation}
				
				$$+\bigg(\sum_{\substack{ \sum_{l=1}^{k-2}i_{l}=k-2,
				0\leq i_l \leq k-1-l }} n_{i_1,...,i_l,...,i_{k-2}}^{(k-1)}f_{k-1}\partial_{x}\prod_{l=1}^{k-2}f_{l}^{(i_l)}\bigg)\partial_{x}^{(1)}f    $$
			
			  $$+ \sum_{j=2}^{k-1}\bigg(\sum_{\substack
			  	{\sum_{l=1}^{k-2}i_{l}=k-1-j,
			  		0\leq i_l \leq k-1-l }} n_{i_1,...,i_l,...,i_{k-2}}^{(k-1)}f_{k-1}^{(1)}\prod_{l=1}^{k-2}f_{l}^{(i_l)}\bigg)\partial_{x}^{(j)}f$$
		  		$$+\sum_{j=2}^{k-1}\bigg(\sum_{\substack
		  		{\sum_{l=1}^{k-2}i_{l}=k-1-j,
		  		0\leq i_l \leq k-1-l}} n_{i_1,...,i_l,...,i_{k-2}}^{(k-1)}f_{k-1}\partial_{x}\prod_{l=1}^{k-2}f_{l}^{(i_l)}\bigg)\partial_{x}^{(j)}f      $$
	  		$$ +\sum_{j=2}^{k-1}\bigg(\sum_{\substack
	  			{\sum_{l=1}^{k-2}i_{l}=(k-1-j)+1,
	  				0\leq i_l \leq k-1-l }} n_{i_1,...,i_l,...,i_{k-2}}^{(k-1)}f_{k-1}\prod_{l=1}^{k-2}f_{l}^{(i_l)}\bigg)\partial_{x}^{(j)}f $$
  				$$+n_{\underbrace{0,...,0}_{(k-2)-times}}^{(k-1)}\prod_{l=1}^{k-1}f_{l}^{(i_l)}\partial_{x}^{(k)}f. $$
  				
  				Let's consider the first two members of \eqref{eq:Equation_10} and calculate $f_{k-1}\partial_x\prod_{l=1}^{k-2}f_{l}^{(i_l)}$ which, with  $f_{k-1}^{(1)}\prod_{l=1}^{k-2}f_{l}^{(i_l)}$, we can generalise in this way $f_{l'}^{(i_{l'}+1)}\prod_{l=1, l\neq l'}^{k-2}f_{l}^{(i_l)}$. Now we can factor out the first two terms $\partial_{x}^{(1)}f$, and furthermore we can gather the common factors $f_{k-1}\partial_x\prod_{l=1}^{k-2}f_{l}^{(i_l)}$, i.e
                
                \begin{equation}\label{eq:Equation_20}
                	\sum_{\sum_{l=1}^{k-2}\tilde{i_l}=k-2, 0\leq \tilde{i_l}\leq k-1-l }\sum_{\hspace{0.5cm}l'=1, (\tilde{i}_1 ,...,\tilde{i}_{l'}+1  ,...,\tilde{i}_l ,...,\tilde{i}_{k-1})=(i_1,...,i_l,...,i_{k-1})}^{k-1}n_{\tilde{i}_1,...,\tilde{i}_l ,...,\tilde{i}_{k-2}}^{(k-1)}f_{l'}^{(\tilde{i}_{l'}+1)}\prod_{l=1, l\neq l'}^{k-2}f_{l}^{(\tilde{i}_l)}=
                \end{equation}
	$$  = \Bigg(	\sum_{\sum_{l=1}^{k-2}\tilde{i_l}=k-2, 0\leq \tilde{i_l}\leq k-1-l } \sum_{\hspace{0.5cm}l'=1,(\tilde{i}_1 ,...,\tilde{i}_{l'}+1  ,...,\tilde{i}_l ,...,\tilde{i}_{k-1})=(i_1,...,i_l,...,i_{k-1})}^{k-1}n_{\tilde{i}_1,...,\tilde{i}_l ,...,\tilde{i}_{k-2}}^{(k-1)} \Bigg)\prod_{l=1}^{k-1}f_{l}^{(i_l)} .                            $$
	
   Posing $\tilde{i}_{k-1}=0$, where $(i_1,...,i_l,...,i_{k-1})$ is such that $\sum_{l=1}^{k-1}i_l=\sum_{l=1}^{k-2}\tilde{i_l}+1=k-2 +1 =k-1$ and $ 0\leq i_l \leq k-1-l < k-l$, and, since $i_{l'}=\tilde{i}_{l'}+1 $, $0 \leq i_{l'} \leq k-1-l'+1=k-l'$, in particular if $l'=k-1$ we have that $0 \le i_{k-1} \le 1$. 
	\quad
	
	By assuming
	
	\begin{equation}\label{eq:Equation_21}
	n_{i_1,...,i_l,...,i_{k-1}}^{(k)}=	\sum_{\sum_{l=1}^{k-2}\tilde{i_l}=k-2, 0\leq \tilde{i_l}\leq k-1-l } \sum_{\hspace{0.5cm}l'=1,(\tilde{i}_1 ,...,\tilde{i}_{l'}+1  ,...,\tilde{i}_l ,...,\tilde{i}_{k-1})=(i_1,...,i_l,...,i_{k-1})}^{k-1}n_{\tilde{i}_1,...,\tilde{i}_l ,...,\tilde{i}_{k-2}}^{(k-1)}.
	\end{equation}
	From \eqref{eq:Equation_21} it can see that $	n_{i_1,...,i_l,...,i_{k-1}}^{(k)}$ satisfies the proprerties of \eqref{eqs:Equations_2}. In fact, if $i_{k-1}=1$, we have a unique 
	$n_{\tilde{i}_1 ,...,\tilde{i}_l ,...,\tilde{i}_{k-1}}$ such that $(\tilde{i}_1 ,...,\tilde{i}_{l'}+1  ,...,\tilde{i}_l ,...,\tilde{i}_{k-1})=(i_1,...,i_l,...,i_{k-1})$. We can still recall, this from the fact that when moving from the $\partial_{x}D_{f_{l},1}^{(k-2)}f$ to $\partial_{x}D_{f_{l},1}^{(k-1)}f$ we have to multiply for an opportune $f_l$, in this case $l=k-1$ and derive a single function, once again, in this case $f_{k-1}$.
	For the case $i_{k-1}$ if $\mathcal{I}max\big((i_l)_{l=1}^{k-1}\big)=\{ l\}$ with $l \neq k-1$, the reasoning and the conclusion here are similar to those used previously; in this line of thinking, the index $i_l$ must be substituted for $i_{k-1}$, and here too, the equation \eqref{eqs:Equations_2} is satisfied. If $|\mathcal{I}max\big((i_l)_{l=1}^{k-1}\big)|=0$ we have more than one solution  $n_{\tilde{i}_1 ,...,\tilde{i}_l ,...,\tilde{i}_{k-1}}$ such that $(\tilde{i}_1 ,...,\tilde{i}_{l'}+1  ,...,\tilde{i}_l ,...,\tilde{i}_{k-1})=(i_1,...,i_l,...,i_{k-1})$ and we have proven \eqref{eqs:Equations_2} for $\partial_{x}^{(1)}f$.
	
	\quad
	We consider, for an fixed $j \neq 1$, we can collect the terms $f_{k-1}\prod_{l=1}^{k-2}f_{l}^{(i_l)}$ of $\partial_{x}^{(j)}f$ and from  \eqref{eq:Equation_10} we get, always assuming $\tilde{i}_{k-1}=0$, 
	
	\begin{equation}\label{eq:Equation_23}
		\Bigg(	\sum_{\sum_{l=1}^{k-2}\tilde{i_l}=k-1-j, 0\leq \tilde{i_l}\leq k-1-l } \sum_{\hspace{0.5cm}l'=1,(\tilde{i}_1 ,...,\tilde{i}_{l'}+1  ,...,\tilde{i}_l ,...,\tilde{i}_{k-1})=(i_1,...,i_l,...,i_{k-1})}^{k-1}n_{\tilde{i}_1,...,\tilde{i}_l ,...,\tilde{i}_{k-2}}^{(k-1)}+ n_{i_1,...,i_l,...,i_{k-2}}^{(k-1)} \Bigg)f_{k-1}\prod_{l=1}^{k-2}f_{l}^{(i_l)}, 
	\end{equation}
	 where $\sum_{l=1}^{k-2}i_{l}=(k-1-j)+1=k-j,
	 0\leq i_l \leq k-1-l$. 
	 \quad
	 If we pose 
	 
	 \begin{equation}\label{eq:Equation_24}
	 	n_{i_1,...,i_l,...,i_{k-1}}^{(k)}=\sum_{\sum_{l=1}^{k-2}\tilde{i_l}=k-1-j, 0\leq \tilde{i_l}\leq k-1-l } \sum_{\hspace{0.5cm}l'=1,(\tilde{i}_1 ,...,\tilde{i}_{l'}+1  ,...,\tilde{i}_l ,...,\tilde{i}_{k-1})=(i_1,...,i_l,...,i_{k-1})}^{k-1}n_{\tilde{i}_1,...,\tilde{i}_l ,...,\tilde{i}_{k-2}}^{(k-1)}+ n_{i_1,...,i_l,...,i_{k-2}}^{(k-1)}
	 \end{equation}
	  this equation satisfies \eqref{eqs:Equations_2} in the case $|\mathcal{I}max\big((i_l)_{l=1}^{k-1}\big)|=0$.
	  Then there is the case, always associated with $\partial_{x}^{(j)}f$, of the collection of elements $f_{k-1}^{(1)}\prod_{l=1}^{k-2}f_{l}^{(i_l)}$, and we get from \eqref{eq:Equation_10} a unique solution $n_{i_1,...,i_l,...,i_{k-2}}^{(k-1)}$. This case satisfes the condition $\mathcal{I}max\big((i_l)_{l=1}^{k-1}\big)=\{k-1\}$ of \eqref{eqs:Equations_2}. 
	  In any case, both cases can be combined in the following, as in the case where $j=1$
	  
	  \begin{equation}\label{eq:Equation_25}
	  	n_{i_1,...,i_l,...,i_{k-1}}^{(k)}=	\sum_{\sum_{l=1}^{k-2}\tilde{i_l}=k-1-j, 0\leq \tilde{i_l}\leq k-1-l } \sum_{\hspace{0.5cm}l'=1,(\tilde{i}_1 ,...,\tilde{i}_{l'}+1  ,...,\tilde{i}_l ,...,\tilde{i}_{k-1})=(i_1,...,i_l,...,i_{k-1})}^{k-1}n_{\tilde{i}_1,...,\tilde{i}_l ,...,\tilde{i}_{k-2}}^{(k-1)},
	  \end{equation}
  where $\sum_{l=1}^{k-1}i_{l}=k-j$,
  $0\leq i_l \leq k-l$, in particular equation \eqref{eq:Equation_25} is the general case of \eqref{eq:Equation_21} . 
  And finally, by posing
   $$n_{\underbrace{0,...,0}_{(k-1)-times}}^{(k)} =n_{\underbrace{0,...,0}_{(k-2)-times}}^{(k-1)},$$ 
   we note that this equation satisfies \eqref{eqs:Equations_2} for condition $|\mathcal{I}max\big((i_l)_{l=1}^{k-1}\big)|=0$ .
   
   \quad
   To conclude the proof. We can derive from the equations \eqref{eq:Equation_10} and \eqref{eq:Equation_25} that
   
   $$\hspace{-2.4cm}\partial_{x}D_{f_{l},1}^{(k-1)}f=\sum_{j=1}^{k}\bigg(\sum_{\substack
   	{\sum_{l=1}^{k-1}i_{l}=k-j,
   		0\leq i_l \leq k-l }}\bigg(\hspace{0.2cm} \sum_{\sum_{l=1}^{k-2}\tilde{i_l}=k-1-j, 0\leq \tilde{i_l}\leq k-1-l } \sum_{\hspace{0.5cm}l'=1,(\tilde{i}_1 ,...,\tilde{i}_{l'}+1  ,...,\tilde{i}_l ,...,\tilde{i}_{k-1})=(i_1,...,i_l,...,i_{k-1})}^{k-1}n_{\tilde{i}_1,...,\tilde{i}_l ,...,\tilde{i}_{k-2}}^{(k-1)}\bigg)\prod_{l=1}^{k-1}f_{l}^{(i_l)}\bigg)\cdot$$
   $$\cdot \partial_{x}^{(j)}f=$$
   $$=\sum_{j=1}^{k}\bigg(\sum_{\substack
   	{\sum_{l=1}^{k-1}i_{l}=k-j,
   		0\leq i_l \leq k-l }} n_{i_1,...,i_l,...,i_{k-1}}^{(k)}\prod_{l=1}^{k-1}f_{l}^{(i_l)}\bigg)\partial_{x}^{(j)}f.$$

\end{proof}

\begin{remark}
	We can see that the coefficients of \eqref{eq:Equation_9}
	\begin{equation}\label{eq:Equation_22}
	 n_{i_{1},\underbrace{0,...,0}_{(k-2)-times}}^{(k)}=\binom{k-1}{i_1}\hspace{0.1cm}with \hspace{0.1cm}i_1\neq 0.
	\end{equation}
 Let's prove by induction on $k$. For $k=2$ the base of induction we get $$n_{1}^{(2)}=1=\binom{1}{1}=\binom{k-1}{i_1}$$ for $i_1=1$ and $k=2$.
 Then let's move on to the inductive hypothesis, i.e
 $$n_{i_{1},\underbrace{0,...,0}_{(k-3)-times}}^{(k-1)}=\binom{k-2}{i_1}.$$
 Since the subscription of $n_{i_1,...,i_l,...,i_{k-1}}^{(k)} $ of \eqref{eq:Equation_9} are bound by costraints $\sum_{l=1}^{k-1}i_{l}=k-j,
 0\leq i_l \leq k-l$ necessarily $i_1=k-1$ or $i_1=k-j$ with $j>1$.
 Let's see the first case. For the \eqref{eqs:Equations_2}, we have that
 $$n_{k-1,\underbrace{0,...,0}_{(k-2)-times}}^{(k)}=n_{k-2,\underbrace{0,...,0}_{(k-3)-times}}^{(k-1)}=\binom{k-2}{k-2}=1=\binom{k-1}{k-1}.$$
 In the second case again for \eqref{eqs:Equations_2} we get
   $$n_{k-j,\underbrace{0,...,0}_{(k-2)-times}}^{(k)}=n_{k-j,\underbrace{0,...,0}_{(k-3)-times}}^{(k-1)}+n_{k-j-1,\underbrace{0,...,0}_{(k-3)-times}}^{(k-1)}=\binom{k-2}{k-j}+\binom{k-2}{k-j-1}=\binom{k-1}{k-j}.$$
  
\end{remark}	
	
	 What we have just demonstrated about the coefficients \eqref{eq:Equation_22} reveals that Leibniz's generalized formula applied to the product $f_1\partial_{x}f=D_{f_1}f$, i.e.
	 
	 $$ \partial_{x}^{(k-1)}f_1\partial_{x}f=\sum_{j=0}^{k-1}\binom{k-1}{j}f_{1}^{(j)}\partial_{x}^{k-j}f$$
	  
	 is a particular case of a nested derivative of order $k$ when $f_{k-1}=f_{k-2}=...=f_{k-j}=...=f_{2}=1$.
	 
	 \quad
	 
	 Given that our object of study is nonlinear polynomial operators in $u$, that is, $	\mathcal{O}_x u \in \mathcal{C}^{\infty}(U;\mathbb{C})\big[\partial^{(k)}_{x}u \big]_{k=0}^{l}$, then we'll study the nested derivative $\partial_{x}D_{f_{l},1}^{(k-1)}P(u)$ applied to $P(u) \in \mathbb{C}[u]$. Let's begin to compute explicitly $ \partial_{x}D_{f_{l},1}^{(k-1)}P(u)$. In order to do this, in addition to \eqref{eq:Equation_9}, we make use of Faa' di Bruno's Formula on $P(u)$(see \cite{faa1855sullo}).
	 
	 \begin{equation}\label{eq:Equation_11}
	 \partial_{x}D_{f_{l},1}^{(k-1)}P(u)=\sum_{j=1}^{k}\bigg(\sum_{\substack
	 	{\sum_{l=1}^{k-1}i_{l}=k-j,
	 		0\leq i_l \leq k-l }} n_{i_1,...,i_l,...,i_{k-1}}^{(k)}\prod_{l=1}^{k-1}f_{l}^{(i_l)}\bigg)\partial_{x}^{(j)}P(u)= 
 		\end{equation}
 		
 		$$=\sum_{j=1}^{k}\bigg(\sum_{\substack
 			{\sum_{l=1}^{k-1}i_{l}=k-j,
 				0\leq i_l \leq k-l }} n_{i_1,...,i_l,...,i_{k-1}}^{(k)}\prod_{l=1}^{k-1}f_{l}^{(i_l)}\bigg)\sum_{i=1}^{j}\partial_{u}^{(i)}P(u)B_{j,i}\big((u^{(i)}(x,t))_{i=1}^{j-i+1}\big),$$
 			
 			\quad
 			where $B_{j,i}\big((u^{(i)}(x,t))_{i=1}^{j-i+1}\big)$ is Bell's polynomial (for this see \cite{bell1927partition}) applied to partial derivatives of $u(x)$ i.e.
 			
 			$$B_{j,i}\big((u^{(i)}(x,t))_{i=1}^{j-i+1}\big)=\sum_{\sum_{l=1}^{j-i+1}j_{l}=i,\sum_{l=1}^{j-i+1}lj_{l}=j }\frac{j!}{j_{1}!j_{2}!...j_{j-i+1}!}\prod_{l=1}^{j-i+1}\Bigg(\frac{u^{(l)}}{l!}\Bigg)^{j_{l}}.$$
 			
 			Now if $P(u)=\sum_{s=0}^{m}b_s u^{s}(x,t)$, we get 
 			
 			$$\sum_{i=1}^{j}\partial_{u}^{(i)}P(u)B_{j,i}\big((u^{(i)}(x,t))_{i=1}^{j-i+1}\big)=\sum_{i=1}^{j}\sum_{s=i}^{m}b_{s}\prod_{t=0}^{i-1}(s-t)u^{s-i}\sum_{\sum_{l=1}^{j-i+1}j_{l}=i,\sum_{l=1}^{j-i+1}lj_{l}=j }\frac{j!}{j_{1}!j_{2}!...j_{j-i+1}!}\prod_{l=1}^{j-i+1}\Bigg(\frac{u^{(l)}}{l!}\Bigg)^{j_{l}}.$$
 			
 			And thus \eqref{eq:Equation_11} becomes
 			
 			\begin{equation}\label{eq:Equation_12}
 				\partial_{x}D_{f_{l},1}^{(k-1)}P(u)=
 			\end{equation}

 			$$\hspace{-2cm}=\sum_{j=1}^{k}\bigg(\sum_{\substack
 				{\sum_{l=1}^{k-1}i_{l}=k-j,
 					0\leq i_l \leq k-l }} n_{i_1,...,i_l,...,i_{k-1}}^{(k)}\prod_{l=1}^{k-1}f_{l}^{(i_l)}\bigg)\sum_{i=1}^{j}\sum_{s=i}^{m}b_{s}\prod_{t=0}^{i-1}(s-t)u^{s-i}\sum_{\sum_{l=1}^{j-i+1}j_{l}=i,\sum_{l=1}^{j-i+1}lj_{l}=j }\frac{j!}{j_{1}!j_{2}!...j_{j-i+1}!}\prod_{l=1}^{j-i+1}\Bigg(\frac{u^{(l)}}{l!}\Bigg)^{j_{l}},$$
 				
 				\quad
 				
 				that by shifting the powers of $u^{s-i} $ towards the products of the powers of the partial derivatives of $u$  i.e $\prod_{l=1}^{j-i+1}\Bigg(\frac{u^{(l)}}{l!}\Bigg)^{j_{l}}$, we obtain
 				
 				\begin{equation}\label{eq:Equation_13}
 						\partial_{x}D_{f_{l},1}^{(k-1)}P(u)=
 						\end{equation}
 					
 					$$	
 				\hspace{-2cm}=	\sum_{j=1}^{k}\sum_{i=1}^{j}\sum_{s=1}^{m}\sum_{\sum_{l=1}^{j-i+1}j_{l}=i,\sum_{l=1}^{j-i+1}lj_{l}=j }\Bigg[\sum_{\substack
 							{\sum_{l=1}^{k-1}i_{l}=k-j,
 								0\leq i_l \leq k-l }}n_{i_1,...,i_l,...,i_{k-1}}^{(k)}b_{s}\prod_{t=0}^{i-1}(s-t)\frac{j!}{j_{1}!{2!}^{j_{2}}j_{2}!...{(j-i+1)!}^{j_{j-i+1}}j_{j-i+1}!}\prod_{l=1}^{k-1}f_{l}^{(i_l)}\Bigg]\cdot$$
 							$$\cdot u^{s-i} \prod_{l=1}^{j-i+1}\big(u^{(l)}\big)^{j_{l}}. $$

\quad

From \eqref{eq:Equation_13} we can see explicitly as $	\partial_{x}D_{f_{l},1}^{(k-1)}P(u)$ differential non-linear operator polynomial respect with to $u$, when $m>1$ i.e $\partial_{x}D_{f_{l},1}^{(k-1)}P(u) \in \mathcal{C}^{\infty}(U;\mathbb{C})\big[\partial^{(k)}_{x}u \big]_{k=0}^{l}$ with bases $u^{s-i} \prod_{l=1}^{j-i+1}\big(u^{(l)}\big)^{j_{l}}$ and coefficients $$\Bigg[\sum_{\substack
	{\sum_{l=1}^{k-1}i_{l}=k-j,
		0\leq i_l \leq k-l }}n_{i_1,...,i_l,...,i_{k-1}}^{(k)}\prod_{l=1}^{k-1}f_{l}^{(i_l)}b_{s}\prod_{t=0}^{i-1}(s-t)\frac{j!}{j_{1}!{2!}^{j_{2}}j_{2}!...{(j-i+1)!}^{j_{j-i+1}}j_{j-i+1}!}\Bigg].$$ 			
			
\quad

To compact the notation of \eqref{eq:Equation_13}, we want to rewrite this equation as \eqref{eq:Equation_3}, i.e.

\begin{equation}\label{eq:Equation_14}
		\partial_{x}D_{f_{l},1}^{(k-1)}P(u)=\sum_{I(s,i,j) \in \mathcal{D}_{m}^{(k)}}\hat{\mathcal{D}}_{I(s,i,j)}^{(k)}(f_{l}, b_{s})\mathcal{P}_{k,x}^{I(s,i,j)}u,
\end{equation}

where $$\mathcal{D}_{m}^{(k)}:=\{ I(s,i,j) \in \mathbb{N}^{k+1} | 1\leq s \leq m, 1\leq i \leq j, 1 \leq j \leq k \},	$$	
 			
and  $$I(s,i,j)	=(i_{0}(s,i,j),..., i_{l}(s,i,j),...,i_{k}(s,i,j))	$$ with $	i_{0}(s,i,j)=s-i$ and $$i_{l}(s,i,j)=\begin{cases}
	j_{l} \\
	0 \hspace{0.2cm}if\hspace{0.2cm} l>j-i+1
\end{cases}$$ 
 			
 such that 	$\sum_{l=1}^{j-i+1}j_{l}=i,\sum_{l=1}^{j-i+1}lj_{l}=j $.

And then
 $$\hat{\mathcal{D}}_{I(s,i,j)}^{(k)}(f_{l}, b_{s}):=\Bigg[\sum_{\substack
{\sum_{l=1}^{k-1}i_{l}=k-j,
	0\leq i_l \leq k-l }}n_{i_1,...,i_l,...,i_{k-1}}^{(k)}\prod_{l=1}^{k-1}f_{l}^{(i_l)}b_{s}\prod_{t=0}^{i-1}(s-t)\frac{j!}{j_{1}!{2!}^{j_{2}}j_{2}!...{(j-i+1)!}^{j_{j-i+1}}j_{j-i+1}!}\Bigg].$$ 

And furthermore $$ \sum_{I(s,i,j) \in \mathcal{D}_{m}^{(k)}}=\sum_{j=1}^{k}\sum_{i=1}^{j}\sum_{s=1}^{m}\sum_{\sum_{l=1}^{j-i+1}j_{l}=i,\sum_{l=1}^{j-i+1}lj_{l}=j }$$.	

\quad		
 \section{Criteria for determining whether an operator $\mathcal{O}_x u$  is an algebraic combination of nested derivatives}\label{se:Section_4}	
 
 Now we want to find a criterion to verify if an operator 	$\mathcal{O}_x u \in \mathcal{C}^{\infty}(U;\mathbb{C})\big[\partial^{(k)}_{x}u \big]_{k=0}^{l}$ can break down into terms of nested derivatives, i.e   	$\mathcal{O}_x u \in (\partial_{x}D_{f_{l},1}^{(k-1)}P(u))_{k=2}^{n}$ for opportune functions $f_{l}(x) \in \mathcal{C}^{\infty}(U;\mathbb{C})$ with $2\leq l \leq n$ and an opportune polynomial $P(u) \in \mathbb{C}[u]$.
 
 \quad
 
 If there's a finite sequence of point $\mathcal{F}$ in $\mathbb{N}^{l'+1}$ associated with $\mathcal{O}_x u$, i.e
 
 $$
 \mathcal{O}_x u=\sum_{I \in \mathcal{F}} A_I (x,t) \mathcal{P}_{l',x}^{I}u,$$
 
such that it exists a family of sub-successions $\{ \mathcal{F}^{(j')} \}$ such that $\mathcal{F}=\bigsqcup^{t'}_{j'=1}\mathcal{F}^{(j')}$ where
 
 \begin{equation}\label{eq:Equation_15}
 	|\mathcal{F}^{(j')}|=|\mathcal{D}_{m}^{(k_{j'})}| 	
 \end{equation}
	
and it exists a bijective correspondence  $\phi^{(j')}$

$$ \phi^{(j')}: \mathcal{D}_{m}^{(k_{j'})}\longrightarrow  \mathcal{F}^{(j')}$$
$$                I(s,i,j) \longrightarrow I $$

such that $\forall I \in \mathcal{F}^{(j')}$
\quad

\begin{equation}\label{eq:Equation_16}
	 A_I (x,t)=a_{j'}(x,t)\hat{\mathcal{D}}_{I(s,i,j)}^{(k_{j'})}(f_{l}, b_{s}),
\end{equation}

\begin{equation}\label{eq:Equation_17}
	\mathcal{P}_{l',x}^{I}u=\prod_{j''=0}^{l''}\big(\partial_{x}^{(j'')}u\big)^{i_{j''}}\mathcal{P}_{k_{j'},x}^{I(s,i,j)}u
\end{equation}

if $\phi^{(j')}(I(s,i,j))=I$.

\quad

Such a family can be called a $nested$  $partition$, if it satisfies \eqref{eq:Equation_15}, \eqref{eq:Equation_16}, and \eqref{eq:Equation_17}.

So we can see that 

$$\mathcal{O}_x u=\sum_{I \in \mathcal{F}} A_I (x,t) \mathcal{P}_{l',x}^{I}u=$$

$$=\sum_{j'=1}^{t'}\sum_{I \in \mathcal{F}^{(j')} }A_I (x,t)\mathcal{P}_{l',x}^{I}u=\sum_{j'=1}^{t'}\sum_{I(s,i,j) \in \mathcal{D}_{m}^{(k_{j'})}  }a_{j'}(x,t)\hat{\mathcal{D}}_{I(s,i,j)}^{(k_{j'})}(f_{l}, b_{s})\prod_{j''=0}^{l''}\big(\partial_{x}^{(j'')}u\big)^{i_{j''}}\mathcal{P}_{k_{j'},x}^{I(s,i,j)}u= $$

$$=\sum_{j'=1}^{t'}a_{j'}(x,t)\prod_{j''=0}^{l''}\big(\partial_{x}^{(j'')}u\big)^{i_{j''}}\sum_{I(s,i,j) \in \mathcal{D}_{m}^{(k_{j'})}}\hat{\mathcal{D}}_{I(s,i,j)}^{(k_{j'})}(f_{l}, b_{s})\mathcal{P}_{k_{j'},x}^{I(s,i,j)}u= $$

$$=\sum_{j'=1}^{t'}a_{j'}(x,t)\prod_{j''=0}^{l''}\big(\partial_{x}^{(j'')}u\big)^{i_{j''}}	\partial_{x}D_{f_{l},1}^{(k_{j'}-1)}P(u),                                            $$
\quad
i.e, $$\mathcal{O}_x u \in ( \partial_{x}D_{f_{l},1}^{(k-1)}P(u)  )_{k=2}^{\overline{k}},$$ with an opportune $\overline{k}\geq k_{j'}$ for $ 1 \leq j' \leq t'$ and an opportune polynomial $\sum_{s=0}^{m}b_{s}u^{s}$ and functions $f_{l}(x)$ with $1 \leq l \leq max_{1\leq j' \leq t'} \{k_{ j'}\}$.

\quad
From these facts we can find a criteria to determine if an operator $\mathcal{O}_x u$ is an algebraic combination of nested derivative of the type $\partial_{x}D_{f_{l},1}^{(k-1)}P(u)$, and it consists of verifying if there exsits a sequence of multi-index that characterizes $\mathcal{O}_x u$, i.e, $\mathcal{F}$ (see \eqref{eq:Equation_3} ) 
may have a nested partition, i.e, if it satisfies \eqref{eq:Equation_15}, \eqref{eq:Equation_16}, \eqref{eq:Equation_17}. This implies that, once condition \eqref{eq:Equation_15} has been met, solving differential equations \eqref{eq:Equation_16}, and \eqref{eq:Equation_17}. This is not always straightforward, even with just one function $f_{l}$ for \eqref{eq:Equation_16}.

\quad

Let’s look at an example of a non-linear differential operator to which we apply the criterion we have just discussed.

\begin{example}\label{ex:Example_1}
Given a differential operator
\begin{equation}\label{eq:Equation_18}
 \mathcal{O}_x u=n(n-1)\sinh x u^{n-2}(\partial_x u)^{2}+n\sinh x u^{n-1}\partial_{x}^{(2)}u+ n\cosh x u^{n-1}\partial_{x}u,
\end{equation}
with a positive integer $n>1.$ From \eqref{eq:Equation_18}, we can guess that the operator could be associated to $\partial_{x}D_{f_{l},1}^{1}P(u)$ with $deg_{u}P(u)=n$. Applying the criteria that we have just seen to equation \eqref{eq:Equation_18}, we consider the following equations

$$n(n-1)\sinh x=n_{0}^{(2)}n(n-1)b_{n}\frac{2!}{1!2!}f_{1},$$

associated with $u^{n-2}(\partial_x u)^{2}.$

$$n\sinh x=n_{0}^{(2)}nb_{n}\frac{2!}{0!1!1!2!}f_{1},$$

associated to $u^{n-1}\partial_{x}^{(2)}u.$

$$ n\cosh x=n_{1}^{(2)}n\frac{1!}{1!1!}f_{1}^{(1)},$$

associated to $u^{n-1}\partial_{x}u.$

Since $=n_{0}^{(2)}=n_{1}^{(2)}=1$ (see \eqref{eqs:Equations_2}), the solution of this system of equations are $b_{n}=1$, $f_{1}=\sinh x$.

\quad

Then, given that no other coefficients $b_{s}$ apart from $b_{n}$, we obtain that
$P(u)=u^{n}$ so we are proved that

$$ \mathcal{O}_x u=\partial_{x}D_{\sinh x,1}^{1}u^{n}=\partial_{x}\sinh x \partial_{x} u^{n}.$$
\end{example}

\section{Criteria for determining whether elementary solutions exist for PDE $ \partial_{x}D_{f_{l},1}^{(k-1)}P(u)=0$ and applications to nonlinear PDEs in one-dimensional space-time}\label{se:Section_5}	

\quad
First of all, we must make it clear what is meant by elementary function. To make this, we must have a brief introduction to differential algebra. In particular, we consider the differential fields, i.e, $(K,d)$ where $K$, is a field and $d$ is a linear application over a subfield, called the field of constant $C \subseteq K$, such that $\forall c \in C$ and $d(c)=0$ and furthermore $d$ satisfies the Leibniz principle $d(hg)=d(h)g+hd(g)$. A simple differential extension of fields $K \subset K(g)$ is called elementary if it is exponential $d(g)= d(s)g$ for some $s \in K$ or logarithmic $d(g)=\frac{d(s)}{s}$ with $ s \in K$ or algebraic (i.e, $g$ is a solution of a polynomial equation irreducible in $K$).
For the purpose of this article, we consider as elements of opportune fields the numeric functions $f(x)$ of the real variable $x$ and $d= \partial_{x}$ the usual derivative, and we shall take the field of rational algebraic functions as our initial field $\mathbb{C}(x)$.
With these tools, we can say what we mean by elementary functions.

\begin{definition}\label{df:Definition_2}
A function $f(x) \in K$ where $K$ is a differential field is an elementary function if it exists a finite chain of elementary extensions of fields that begins with $\mathbb{C}(x)$ and ends with $K$

$$\mathbb{C}(x) \subset K_1 \subset ...\subset K_l\subset ....\subset K_t=K$$
	
\end{definition} 

Elementary functions, for example, are the functions studied in high school, e.g, $\cos(x)$, $e^{x}$, $\sqrt{x-3}$ etc...

\quad

In this type of language, finding primitive families of a function $f$ using indefinite integrals amounts to determining a function $g$ such that $(g(x))'=f(x)$. In particular, we want to know when $g$ is an elementary function. Liouville's theorem comes to our aid here, stating that $g$ is elementary if 
\begin{equation}\label{eq:Equations_26}
	f(x)=\sum_{i=1}^{n}c_i \frac{(h_i(x))'}{h_i(x)}+ (s(x))',
\end{equation}
where $c_i \in \mathbb{C}$ and $h_i(x)$, $s(x)$ are elementary functions. However, if the reader wishes to explore this topic in greater depth, we recommend consulting the text \cite{de2014teorema}.
\quad

Now let's start finding the solutions of the ODE

\begin{equation}\label{eq:Equation_28}
	\partial_{x}D_{f_{l},1}^{(k-1)}u=0.
\end{equation}

It can be set that a solution of $\partial_{x}f_1 \partial_{x}u=0$ is a solution of $\partial_{x}D_{f_{l},1}^{(k-1)}u=0$, which, fixing an function $f(t)$, we can call such a solution $u_1 = u_1(x,t)$ and $$\partial_{x}f_1 \partial_{x}u=0 \rightarrow f_1 \partial_{x}u_1=f(t).$$
$$\partial_{x}u_1=\frac{f(t)}{f_1(x)}\rightarrow u_1(x,t)=f(t)R_{f_1}(x).$$

\quad 

where 

\begin{equation}\label{eq:Equation_27}
	R_{f_1}(x):= \int_{x_0}^{x}\frac{1}{f_1(s)}ds,
\end{equation}
 with $x_0$ an opportune initial condition of ODE $\partial_{x}D_{f_{l},1}^{(k-1)}u=0.$
 In addition, solutions that meet ODE requirements $\partial_{x}D_{f_{l},1}^{(2)}u=0$, i.e, $\partial_{x}f_2 \partial_x f_1 \partial u(x,t)=0$, are possible solutions of \eqref{eq:Equation_28}. Let's see how we can work them out
 $$ f_2 \partial_{x}f_1 \partial u(x,t)=f(t) $$
 for an opportune choice of $f(t)$.
 
Then 

$$\partial_{x}f_1\partial_{x}u=\frac{f(t)}{f_2(x)} \rightarrow f_1\partial_{x}u=f(t)R_{f_2}(x)$$

where $R_{f_2}(x)$ is defined by \eqref{eq:Equation_27} .

$$ \partial_{x}u =f(t) \frac{R_{f_2}(x)}{f_1 (x)}.$$

Thus the possible solution of \eqref{eq:Equation_28}  which we'll call $u_2(x,t)$ is

$$ u_2(x,t)=f(t)\int_{x_0}^{x}\frac{R_{f_2}(s)}{f_1 (s)}ds,$$

 where, we recall again, $x_0$ is an opportune choice of initial condition of ODE  \eqref{eq:Equation_28}. We can call the solutions $u_j(x,t)$ associated with the ODE $\partial_{x}D_{f_{l},1}^{(j)}u=0$ and determined in the same way as we did in the previous case for $j=1$ and $j=2$. The next step is trying it generalize this procedure. It can be noted that every solution $u_j$ is a solution of $\eqref{eq:Equation_28}$ for $j\leq k-1$.
 Let's start again for $j=3$
 
 $$\partial_{x}D_{f_{l},1}^{(3)}u=0 \rightarrow   D_{f_{l},1}^{(3)}u=f(t)$$
 
 $$\partial_{x}D_{f_{l},1}^{(2)}u=\frac{f(t)}{f_3(x)} \rightarrow D_{f_{l},1}^{(2)}u=f(t) R_{f_3}(x) \rightarrow \partial_{x}D_{f_{l},1}^{(1)}u=f(t)\frac{R_{f_3}(x)}{f_2 (x)} $$
 
 $$ D_{f_{l},1}^{(1)}u=f(t) g_{2:3} (x)$$
 
 where $g_{2:3}:= \int_{x_0}^{x}\frac{R_{f_3}(s)}{f_2 (s)}ds$.
 
 So $u_3$ will be
 
 $$ u_3(x,t)=f(t)\int_{x_0}^{x}\frac{g_{2:3}(s)}{f_1(s)}ds.$$
 
 For $j>3$ we get
 
 \begin{proposition}\label{pr:Proposition_2}
 	Given the ODE
 	
 	$$\partial_{x}D_{f_{l},1}^{(j)}u=0.$$
 	 Where $f_{l}(x)$ are known functions of class at least $C^{j+1} $.
 	 The possible solution of this equation is 
 	 
 	 $$ u_j(x,t)=f(t)\int_{x_0}^{x}\frac{g_{2:j}(s)}{f_1(s)}ds,$$
 	 
 	 where $g_{2:j}(s)$ is a function determined recursively in this way
 	 
 	 $$g_{2:j}=\int_{x_0}^{x}\frac{g_{3:j}(s)}{f_2(s)}ds $$
 	 
 	 $$ g_{i:j}=\int_{x_0}^{x}\frac{g_{i+1:j}(s)}{f_i(s)}ds \hspace{0.5cm} 2 \leq i \leq j-1$$
 	 
 	 $$  g_{j-1:j}= \int_{x_0}^{x}\frac{R_{f_j}(s)}{f_{j-1} (s)}ds,$$
 	
 	where $x_0$ is an opportune initial condition of $ODE$, and $g_{i:j}$ we can call functions $resolving$ of the nested derived equation of order $j+1$.
 \end{proposition}
 \begin{proof}
 	Let's prove by induction on $j$.
 	We prove the base induction $j=4$
 	$$\partial_{x}D_{f_{l},1}^{(4)}u=0 \rightarrow \partial_{x}D_{f_{l},1}^{(3)}u=\frac{f(t)}{f_4(x)} \rightarrow D_{f_{l},1}^{(3)}u=f(t)R_{f_4}(x) \rightarrow \partial_{x}D_{f_{l},1}^{(2)}u=f(t)\frac{R_{f_4}(x)}{f_3(x)}$$
 	
 	$$D_{f_{l},1}^{(2)}u=f(t)g_{3:4}(x) \rightarrow \partial_{x}D_{f_{l},1}^{(1)}u =f(t)\frac{g_{3:4}(x)}{f_2(x)}\rightarrow D_{f_{l},1}^{(1)}u=f(t)g_{2:4}(x) \rightarrow \partial_{x}u=f(t)\frac{g_{2:4}}{f_1(x)} $$
 	$$ u_4(x,t)= f(t)\int_{x_0}^{x}\frac{g_{2:4}(s)}{f_1(s)}ds.$$
 	
 	For inductive hypothesis, we get that
 	
 	$$u_{j-1}=f(t)\int_{x_0}^{x}\frac{g_{2:j-1}(s)}{f_1(s)}ds.$$
 	
 	Let's see how to use this to prove the proposition.
    We can rewrite $\partial_{x}D_{f_{l},1}^{(j)}u=0.$ as $ \partial_{x}D_{\bar{f}_{l},1}^{(j-1)}\bar{u}=0$ where $\bar{u}=D_{f_{l},1}^{(1)}u$  and $\bar{f}_{l}=f_{l+1}$ for inductive hypothesis on $\bar{u}$ we get 
    $$\bar{u}_{j-1}=f(t)\int_{x_0}^{x}\frac{\bar{g}_{2:j-1}(s)}{\bar{f}_1(s)}ds \rightarrow f_1 \partial_{x} u_j =f(t)\int{x_0}^{x}\frac{\bar{g}_{2:j-1}(s)}{\bar{f}_1(s)}ds=f(t)\int{x_0}^{x}\frac{\bar{g}_{2:j-1}(s)}{f_2(s)}ds$$
     with $g_{3:j}(x)=\bar{g}_{2:j-1}(x)$ we get
     
     $$f_1 \partial_{x} u_j=f(t)\int_{x_0}^{x}\frac{g_{3:j}(s)} {f_2(s)}ds \rightarrow f_1 \partial_{x} u_j=f(t)g_{2:j}(x) \rightarrow \partial_{x} u_j=f(t)\frac{g_{2:j}(x)}{f_1(x)} \rightarrow u_j=f(t)\int_{x_0}^{x}\frac{g_{2:j}(s)}{f_1(s)}ds $$
 	
 \end{proof}

\quad

\begin{remark}\label{os:Observation_2}
	In determining the solution of \ref{pr:Proposition_2}, it should be noted that we are implying the existence the an opportune interval $I \subset \mathbb{R}$ where $x_0 \in I$ and this will be implied throughout the article for all equations. As can be seen from the fundamental theorem of calculus, the functions $resolving$ of the nested derived of \ref{pr:Proposition_2} $g_{i:j} \hspace{0.5cm} 2\leq i \leq j$, we can determine in this way
	
	\begin{equation}\label{eqs:Equations_4}
	(g_{2:j})'=	\frac{g_{3:j}}{f_2} \rightarrow f_2 =\frac{g_{3:j}}{(g_{2:j})'} 	
	\end{equation}
	$$(g_{i:j})'=	\frac{g_{i+1:j}}{f_i} \rightarrow f_i =\frac{g_{i+1:j}}{(g_{i:j})'} \hspace{0.5cm}2\leq i \leq j-1$$
		$$(g_{j-1:j})'=	\frac{R_{f_j}}{f_{j-1}} \rightarrow f_{j-1} =\frac{R_{f_j}}{(g_{j-1:j})'}                                                          $$
\end{remark}

\quad
 From \eqref{eqs:Equations_4} we can deduce that in an attempt to resolve the equation 
 
 $$\partial_{x}f_j \partial_{x}...\partial_{x}f_l \partial_{x}...\partial_{x}f_2\partial_{x} f_1\partial_{x}u=0,$$
 
 we can determine the $resolving$ functions $g_{i:j}$ by \eqref{eqs:Equations_4}, and the equation above becomes
 
 $$\partial_{x}\frac{R_{f_j}}{(g_{j-1:j})'} \partial_{x}...\partial_{x}\frac{g_{l+1:j}}{(g_{l:j})'}\partial_{x}...\partial_{x} \frac{g_{3:j}}{(g_{2:j})'}\partial_{x}f_1\partial_{x}u=0.$$
 \quad
 Surely this way to determine a solution is useful to construct some ODEs that have this type of solution; in our article, in particular, elementary functions of the solutions, for examples.
 Anyway this isn't enough to find a solution. We need to calculate 
 
 \begin{equation}\label{eq:Equation_29}
 	u_j(x,t)=f(t)\int_{x_0}^{x}\frac{g_{2:j}(s)}{f_1(s)}ds.
 \end{equation}
  
We can attempt to integrate it by parts; to achieve this, we provide a way of expressing a product of integral by parts, a method that has been known since the early years of studying calculus.

\quad

First, to simplify the notation, we want to denote as $D_{x_0}^{-1}f(x):=\int_{x_0}^{x}f(s)ds$, so in the case $u_j(x,t)=f(t)D_{x_0}^{-1}\frac{g_{2:j}(x)}{f_1(x)}$. Then we will denote $D_{x_0}^{0}f(x)=f(x)$, and we calculate recursively $D_{x_0}^{-k}f(x)=\int_{x_0}^{x}D_{x_0}^{-(k-1)}f(s)ds$.
With this notation we can remark in general that if $j <k$ $\partial_{x}^{j}D_{x_0}^{-k}f(x)= D_{x_0}^{-(k-j)}f(x)$; otherwise we will get $\partial_{x}^{j}D_{x_0}^{-k}f(x)=\partial_{x}^{j-k}f(x).$
 
Now we integrate by parts $F(x)=\int_{x_0}^{x} u(s)v(s)ds$; we get, with new notation just adopted, considering $v(x)=(V(x))'$

\begin{equation}\label{eq:Equation_30}
	\int_{x_0}^{x} u(s)(V(s))'ds=\sum_{j=0}^{k}(-1)^{j}u^{(j)}(x)D_{x_0}^{-j}V(x)-(-1)^{k}\int_{x_0}^{x}u^{(k+1)}(s)D_{x_0}^{-k}V(s)ds,
\end{equation}
 
 for a certain integer $k \geq 0$.
 
 For $k >1 $, it indicates that we get to integrate by parts for $k$-times. This is useful when you want to explicitly calculate the integral by parts, that is, to determine whether the integral can be expressed in terms of elementary functions.
 
 Now let's prove \eqref{eq:Equation_30} for induction.
 
 The base for induction is for $j=0$ is well-known integration by parts.
 Then the inductive hypothesis is 
 
 $$ 	\int_{x_0}^{x} u(s)(V(s))'ds=\sum_{j=0}^{k-1}(-1)^{j}u^{(j)}(x)D_{x_0}^{-j}V(x)-(-1)^{k-1}\int_{x_0}^{x}u^{(k)}(s)D_{x_0}^{-(k-1)}V(s)ds=$$
 
 $$=\sum_{j=0}^{k-1}(-1)^{j}u^{(j)}(x)D_{x_0}^{-j}V(x)-(-1)^{k-1}u^{(k)}(x)D_{x_0}^{-k}V(x)+(-1)^{k-1}\int_{x_0}^{x}u^{(k+1)}(s)D_{x_0}^{-k}V(s)ds=$$
 
 $$=\sum_{j=0}^{k}(-1)^{j}u^{(j)}(x)D_{x_0}^{-j}V(x)-(-1)^{k}\int_{x_0}^{x}u^{(k+1)}(s)D_{x_0}^{-k}V(s)ds.$$
 
 If we assume $\nu(x)=D_{x_0}^{-k}V(x)$ and $S(x)=\int_{x_0}^{x}u^{(k+1)}(s)D_{x_0}^{-k}V(s)ds=\int_{x_0}^{x}u^{(k+1)}(s)\nu(s)ds.$ Then \eqref{eq:Equation_30} becames
 
 \begin{equation}\label{eq:Equation_31}
 	\int_{x_0}^{x} u(s)\nu^{(k+1)}(s)ds=A\big(\sum_{j=0}^{k}(-1)^{j}u^{(j)}(x)\nu^{(k-j)}(x)\big) + B S(x)
 \end{equation}

where
\begin{equation}\label{eqs: Equations_5}
	A=1,\hspace{0.1cm} B=0 \hspace{0.1cm} if \hspace{0.1cm}S(x)=0
\end{equation}
$$A=\frac{1}{1+(-1)^{k}c} ,\hspace{0.1cm} B=0  \hspace{0.1cm} with \hspace{0.1cm} c \neq 0 \hspace{0.1cm}and \hspace{0.1cm} \frac{u^{(k+1)}(x)}{u(x)}=c\frac{\nu^{(k+1)}(x)}{\nu(x)}$$
$$A=1, \hspace{0.1cm} B=-(-1)^{k}\hspace{0.1cm} otherwise.$$

We can deduce that 

\begin{proposition}\label{pr:Proposition_3}
	Given  the integral function $F(x)=\int_{x_0}^{x} u(s)v(s)ds$, if they exist elementary functions $\nu(x)$ and $S(x)$ such that
	$\nu^{(k+1)}(x)=v(x)$ and $(S(x))'= u^{(k+1)}(x)\nu(x)$ for some positive integer $k$ then $F(x)$ is elementary function.	
\end{proposition}
\begin{proof}
	Under this assumption, it suffices to integrate by parts $k$-times $\int_{x_0}^{x} u(s)v(s)ds$ to determine \eqref{eq:Equation_30}, which shows that $F(x)$ is a sum of elementary functions.
\end{proof}

\quad
\begin{corollary}\label{cor:Corollary_1}
	If they exist elementary functions $\nu(x)$ and $S(x)$ such that
	$\nu^{(k+1)}(x)=\frac{1}{f_1(x)}$ and $(S(x))'= g_{2:j}^{(k+1)}(x)\nu(x)$ for some positive integer $k$ then $u_j(x,t)$ is a elementary solutions of \ref{pr:Proposition_2}.
\end{corollary}

\quad

This criterion apparently seems to be more straightforward than Liouville's  \eqref{eq:Equations_26}. But finding the functions such that $\nu^{(k+1)}(x)=\frac{1}{f_1(x)}$ and $(S(x))'= g_{2:j}^{(k+1)}(x)\nu(x)$ generally speaking is not easy.

\quad
In any case, let's summarize the following results that we have just discussed.

\begin{proposition}\label{pr:Proposition_4}
The equation \eqref{eq:Equation_28} 
$$\partial_{x}D_{f_{l},1}^{(k-1)}u=0$$ 
has elementary solutions of the type $u_j(x,t)$ for $2 \leq j \leq k-1$ if $g_{2:j}(x)\frac{1}{f_1(x)}$ proves the Liouville's theorem 
\eqref{eq:Equations_26}. In particular, $u_1(x,t)$ is an elementary solution if $\frac{1}{f_1(x)}$ proves Liouville's theorem 
\eqref{eq:Equations_26}.
\end{proposition}
 
 \quad 
 \begin{proposition}\label{pr:Proposition_5}
 	The equation \eqref{eq:Equation_28} 
 	$$\partial_{x}D_{f_{l},1}^{(k-1)}u=0$$ 
 	has elementary solutions of the type $u_j(x,t)$ for $2 \leq j \leq k-1$ if $g_{2:j}(x)$ and $\frac{1}{f_1(x)}$ check the criterion \eqref{cor:Corollary_1}.
     In particular, $u_1(x,t)$ is an elementary solution if $1$ constant function and $\frac{1}{f_1(x)}$ check the criterion \eqref{cor:Corollary_1}.
 \end{proposition}

Let us now suppose that the ODE $\partial_{x}D_{f_{l},1}^{(k-1)}u=0$ admits $\{u_{j_{s}}\}$ elementary solutions; then for the linearity of ODE we get $u(x,t)=\sum_{s=0}^{t'}A_{s}u_{j_{s}}$ elementary solutions as the coefficients $A_{s}$ vary in $\mathbb{C}$, i.e, $u \in W=<u_{j_{s}}>$ with $W$ vectorial space. Let's see how to apply this to some examples of non-linear PDEs in space-time.

\begin{example}\label{ex:Example_2}
	Given PDE's non-linear 
	      \begin{equation}\label{eq:Equation_32}
	      	\partial_t t \partial_t u=(u^{3}+2u-1)\partial_{x} x \partial_{x} \frac{\ln x}{\cosh x} \partial_{x} \tanh x \partial_{x} \frac{\sinh x}{ e^x}\partial_{x} \frac{e^{\frac{x}{2}}}{\cos(\frac{\sqrt{3}}{2}x)}u -u .   	
	      \end{equation}
	\quad
	Using the invariant space method, this space $W$ cancels out the nested derivative
	
	$$\partial_{x} x \partial_{x} \frac{\ln x}{\cosh x} \partial_{x} \tanh x \partial_{x} \frac{\sinh x}{ e^x}\partial_{x} \frac{e^{\frac{x}{2}}}{\cos(\frac{\sqrt{3}}{2}x)}u$$
	
	in this case, the functions that characterize the nested derivative are
	
	$f_1=\frac{e^{\frac{x}{2}}}{\cos(\frac{\sqrt{3}}{2}x)}$, $f_2= \frac{\sinh x}{ e^x}$, $f_3=\tanh x$, $f_4=\frac{\ln x}{\cosh x}$, $f_5= x$. 
	\quad
	
	From this it can be deduced that $R_{f_5}=\ln x $, $g_{2:5}=e^{x}$, $g_{3:5}=\sinh x$, $g_{4:5}=\sinh x$. The function $\frac{1}{f_1(x)}=e^{-\frac{x}{2}}\cos(\frac{\sqrt{3}}{2}x)$ is a sort of generalization of $\cos(x)$ associated with primitve roots of order $3$, where it can be seen that $$\bigg(e^{-\frac{x}{2}}\cos \bigg(\frac{\sqrt{3}}{2}x\bigg)\bigg)^{(k)}=e^{-\frac{x}{2}}\cos\bigg( \frac{2k \pi}{3}+\frac{\sqrt{3}}{2}x\bigg).$$
	Thanks to this fact, we get that $u_1(x,t)$ is an elementary solution, since $\bigg(e^{-\frac{x}{2}}\cos \bigg(\frac{4 \pi}{3}+\frac{\sqrt{3}}{2}x\bigg)\bigg)^{'}=e^{-\frac{x}{2}}\cos \bigg(\frac{\sqrt{3}}{2}x\bigg)$(To find out why this is the case, please refer to the section \ref{se:Section_7} of the appendix), so 
	$$u_1(x,t)=C_0(t)\bigg(e^{-\frac{x}{2}}\cos \bigg(\frac{4 \pi}{3}+\frac{\sqrt{3}}{2}x\bigg) +c \bigg),$$
\quad
where $C_0(t)= \sum_{k=0}^\infty \frac{t^k}{k!^2}$ is Bessel type-function such that 	$\partial_t t \partial_t C_0(t)=-C_0(t)$(see \cite{ricci2020laguerre} and \cite{garra2023note} ), and $c$ is a constant that depends of initial conditions of PDE.
For \eqref{cor:Corollary_1} we get that exsists $\nu(x)=e^{-\frac{x}{2}}\cos(\frac{\sqrt{3}}{2}x)$ such that $\nu^{(2+1)}(x)=\frac{1}{f_1(x)}$ and $S(x)=0$, since $\frac{\nu^{(2+1)}(x)}{\nu(x)}=1=\frac{g_{2:5}^{(2+1)}(x)}{g_{2:5}(x)}=\frac{e^{x}}{e^{x}}$, so $u_5(x,t)$ is another elementary solution that verifies \eqref{eq:Equation_31} and \eqref{eqs: Equations_5} for $k=2$, $A=\frac{1}{2}$ and $B=0$

$$u_5(x,t)=\frac{C_0(t)}{2}\bigg(e^{\frac{x}{2}}\bigg(\cos \bigg(\frac{4 \pi}{3}+\frac{\sqrt{3}}{2}x\bigg)-\cos \bigg(\frac{2 \pi}{3}+\frac{\sqrt{3}}{2}x\bigg) + \cos \bigg(\frac{\sqrt{3}}{2}x\bigg)   \bigg)+ c' \bigg)$$.

So we get a family of possible elementary solutions $$u=C_0(t)\bigg(A'\bigg(e^{-\frac{x}{2}}\cos \bigg(\frac{4 \pi}{3}+\frac{\sqrt{3}}{2}x\bigg) \bigg)+ B'\bigg(e^{\frac{x}{2}}\bigg(\cos \bigg(\frac{4 \pi}{3}+\frac{\sqrt{3}}{2}x\bigg)-\cos \bigg(\frac{2 \pi}{3}+\frac{\sqrt{3}}{2}x\bigg) + \cos \bigg(\frac{\sqrt{3}}{2}x\bigg)   \bigg) \bigg) + C'\Bigg)$$
\quad
where $A'$, $B'$, $C'$ are constants and $u=C_0(t)g$ where $g(x) \in W=<g_1, g_5, 1>$ where $u_1=C_0(t)g_1$, $u_5=C_0(t)g_2$ generators of elementary functions of invariant space $W$ for \eqref{eq:Equation_32}.
	
\end{example}
 \subsection{ The non-homogeneous case $ \partial_{x}D_{f_{l},1}^{(k-1)}u=g(x,t)$}\label{sse:Subsection5_1}

\quad

Given the nested derivative equations non-homogeneous 
\begin{equation}\label{eq:Equation_33}
	\partial_{x}D_{f_{l},1}^{(k-1)}u=g(x,t).	
\end{equation}

We want to determine a particular solutions of \eqref{eq:Equation_33}, which we shall call $\bar{u}_{k-1}$ with the analog method \eqref{eq:Equation_28}. Let us begin noting that, once particular solutions $\bar{u}_{k-1}$ have been obtained, the possible solutions include the linear combination of the  $\bar{u}_{k-1}$ and homogeneous solutions by the linearity of the nested derivative operator, i.e.,

$$u(x,t)=\sum_{j=0}^{k-1}A_j u_j(x,t) +B\bar{u}_{k-1}(x,t).$$

\quad
Using the following proposition, we shall see how to determine $\bar{u}_{k-1}$ .

\begin{proposition}\label{pr:Proposition_6}
	The equation \eqref{eq:Equation_33} 
	$$	\partial_{x}D_{f_{l},1}^{(k-1)}u=g(x,t)	$$
	
	admits a particular solutions 
	
	$$\bar{u}_{k-1}(x,t)=\int_{x_0}^{x}\frac{g_{2:k-1}(s,t)}{f_1(s)}ds,$$
	
	if only if they exist the functions $g_{i:k-1}(x,t)$ for $ 1 \leq i \leq k-2$ and $G(x,t), R_{f_{k-1}}(x,t)$ with $\partial_{x} G(x,t)=g(x,t)$ such that 
	
	$$f_i(x)=\frac{g_{i+1:k-1}(x,t)}{\partial_{x}g_{i:k-1}(x,t)} \hspace{0.1 cm} 2 \leq i \leq k-3 $$
	$$f_{k-2}(x)=\frac{R_{f_{k-1}}}{\partial_{x}g_{k-2:k-1}(x,t)}$$
	$$ f_{k-1}(x)=\frac{G(x,t)}{\partial_{x}R_{f_{k-1}}}$$

\end{proposition}

\begin{proof}
Let's prove the first implication by induction of degree of nested derivation.
We begin with the base of induction $k=5$.
\quad

$$\partial_{x}f_4 \partial_{x} f_3 \partial_x f_2 \partial_x f_1 \partial_x u=g(x,t)$$

$$ f_4\partial_{x} f_3 \partial_x f_2 \partial_x f_1 \partial_x u=G(x,t)$$

 with $G(x,t)= \int_{x_0}^{x}g(s,t)ds$ for an opportune conditional condition  $x_0$ of \eqref{eq:Equation_33}. 
$$\partial_{x} f_3 \partial_x f_2 \partial_x f_1 \partial_x u=\frac{G(x,t)}{f_4(x)}$$

If we set $R_{f_{4} }(x,t)=\int_{x_0}^{x}\frac{G(s,t)}{f_4(s)}ds$ .

$$f_3 \partial_x f_2 \partial_x f_1 \partial_x u=R_{f_{4} }$$

$$\partial_x f_2 \partial_x f_1 \partial_x u=\frac{R_{f_{4} }}{f_3} \rightarrow  f_2 \partial_x f_1 \partial_x u=\int_{x_0}^{x}\frac{R_{f_{4} }}{f_3}ds=g_{3:4} \rightarrow \partial_x f_1 \partial_x u=\frac{g_{3:4}}{f_2}.$$

$$ f_1 \partial_x u=g_{2:4} \rightarrow \partial_x u=\frac{g_{2:4}}{f_1} \rightarrow u=\int_{x_0}^{x}\frac{g_{2:4}(s,t)}{f_1(s)}ds,$$

with $g_{2:4}=\int_{x_0}^{x}\frac{g_{3:4}(s,t)}{f_2(s)}ds$.

Now let's start with the inductive hypothesis for $k=l-1$.

$$\partial_{x}f_{l}\partial_{x}...\partial_{i}f_{i}\partial_{x}...\partial_{x}u=g(x,t) \rightarrow \partial_{x}f_{l-1}...\partial_{i}f_{i}\partial_{x}...\partial_{x}u=\int_{x_0}^{x}\frac{G(s,t)}{f_l(s)}ds$$. If we set $\tilde{g}(x,t)=\int_{x_0}^{x}\frac{G(s,t)}{f_l(s)}ds$ the inductive hypothesis holds so exsist $\tilde{g}_{i:l-1}=\int_{x_0}^{x}\frac{\tilde{g}_{i+1:l-1}(s,t)}{f_i(s)}ds$ for $2 \leq i \leq l-3$ and $\tilde{g}_{l-2:l-1}=\int_{x_0}^{x}\frac{\tilde{R}_{f_{l-1}}(s,t)}{f_{l-2}(s)}ds $ with $\tilde{R}_{f_{l-1}}=\int_{x_0}^{x}\frac{\tilde{G}(s,t)}{f_{l-1}(s)}ds$ and $\tilde{G}(s,t)=\int_{x_0}^{x}\tilde{g}(s,t)ds$ such that the solution is
$$ \tilde{u}(x,t)=\int_{x_0}^{x}\frac{\tilde{g}_{2:l-1}(s,t)}{f_1(s)}ds.$$

If we assume that $g_{i:l}=\tilde{g}_{i:l-1}$ for $2 \leq i \leq l-2$  and $g_{l-1:l}=\tilde{R}_{f_{l-1}}$ and $R_{f_{l}}(x,t)=\tilde{g}(x,t)$ we get the solution for $k=l$ i.e

$$u(x,t)=\tilde{u}(x,t)=\int_{x_0}^{x}\frac{\tilde{g}_{2:l-1}(s,t)}{f_1(s)}ds=\int_{x_0}^{x}\frac{g_{2:l}(s,t)}{f_1(s)}ds.$$

The demonstration of inverse implication consists of carrying out the following calculation

$$\partial_{x}\frac{G(x,t)}{\partial_{x}R_{f_{k-1}}}\partial_{x}\frac{R_{f_{k-1}}}{\partial_{x}g_{k-2:k-1}(x,t)} \partial_{x}\frac{g_{k-2:k-1}(x,t)}{\partial_{x}g_{k-3:k-1}(x,t)}\partial_{x}...\partial_{x}\frac{g_{i+1:k-1}(x,t)}{\partial_{x}g_{i:k-1}(x,t)}\partial_{x}...\partial_{x}\frac{g_{3:k-1}(x,t)}{\partial_{x}g_{2:k-1}(x,t)}\partial_{x}f_1(x)\partial_{x}\int_{x_0}^{x}\frac{g_{2:k-1}(s,t)}{f_1(s)}ds.$$

$$\partial_{x}\frac{G(x,t)}{\partial_{x}R_{f_{k-1}}}\partial_{x}......\partial_{x}\frac{g_{i+1:k-1}(x,t)}{\partial_{x}g_{i:k-1}(x,t)}\partial_{x}......\partial_{x}\frac{g_{3:k-1}(x,t)}{\partial_{x}g_{2:k-1}(x,t)}\partial_{x}\bigg(f_1(x)\frac{g_{2:k-1}(x,t)}{f_1(x)}  \bigg)$$

\quad

$$\partial_{x}\frac{G(x,t)}{\partial_{x}R_{f_{k-1}}}\partial_{x}......\partial_{x}\frac{g_{i+1:k-1}(x,t)}{\partial_{x}g_{i:k-1}(x,t)}\partial_{x}......\partial_{x}\frac{g_{3:k-1}(x,t)}{\partial_{x}g_{2:k-1}(x,t)}\partial_{x}g_{2:k-1}(x,t)$$

After repeating this calculation for $i$-steps

$$\partial_{x}\frac{G(x,t)}{\partial_{x}R_{f_{k-1}}}\partial_{x}......\partial_{x}\frac{g_{i+1:k-1}(x,t)}{\partial_{x}g_{i:k-1}(x,t)}\partial_{x}g_{i:k-1}(x,t),$$

we arrive at $(k-2)$-step

$$\partial_{x}\frac{G(x,t)}{\partial_{x}R_{f_{k-1}}}\partial_{x}\frac{R_{f_{k-1}}}{\partial_{x}g_{k-2:k-1}(x,t)}\partial_{x}g_{k-2:k-1}(x,t)=\partial_{x}\frac{G(x,t)}{\partial_{x}R_{f_{k-1}}}\partial_{x}R_{f_{k-1}}=\partial_{x}G(x,t)=g(x,t).$$
Thus, we have shown the converse implication, and we have completed the demonstration.

\end{proof}

\quad
Given the form of the particular solution $\bar{u}_{k-1}$ of \eqref{eq:Equation_33} by \ref{pr:Proposition_6}, also in this case we get two criteria for determining if $\bar{u}_{k-1}$ is elementary in a similar way to the homogeneous case.

\quad

\begin{proposition}\label{pr:Proposition_7}
	The equation \eqref{eq:Equation_33} 
	$$\partial_{x}D_{f_{l},1}^{(k-1)}u=g(x,t)$$ 
	has particular elementary solutions of the type $\bar{u}_{k-1}(x,t)$  if $g_{2:k-1}(x,t)\frac{1}{f_1(x)}$( see \ref{pr:Proposition_6}) proves Liouville's theorem (relative to the coordinate $x$)
	\eqref{eq:Equations_26}. 
	
\end{proposition}	

\quad

\begin{proposition}\label{pr:Proposition_8}
	The equation \eqref{eq:Equation_33} 
	$$\partial_{x}D_{f_{l},1}^{(k-1)}u=g(x,t)$$ 
	has particular elementary solutions of the type $\bar{u}_{k-1}(x,t)$  ( see \ref{pr:Proposition_6})if $g_{2:j}(x,t)$ and $\frac{1}{f_1(x)}$ check the criterion \eqref{cor:Corollary_1}.
\end{proposition}

Now we see examples of non-linear PDEs in space-time that are solved by equation \eqref{eq:Equation_33}.

\begin{example}\label{ex:Example_3}
Given the following non-linear PDE
\begin{equation}\label{eq:Equation_34}
\partial_t t \partial_t u=\frac{u}{\ln(t)e^{-x}(\tanh^{2} x+\tanh x-2)}	\partial_{x} \frac{1}{\cosh x}\partial_{x} \frac{\sinh x}{e^{x}}\partial_{x}\frac{e^{\frac{x}{2}}}{\cos\big( \frac{\sqrt{3}}{2}x \big)}\partial_{x}u -u.
\end{equation}
	
We want to find $u$ such that the first and the second member of \eqref{eq:Equation_34} are cancelled. In particular, we note that the second member is cancelled if $u$ is a solution of

$$	\partial_{x} \frac{1}{\cosh x}\partial_{x} \frac{\sinh x}{e^{x}}\partial_{x}\frac{e^{\frac{x}{2}}}{\cos\big( \frac{\sqrt{3}}{2}x \big)}\partial_{x}u=\ln(t)e^{-x}(\tanh^{2} x+\tanh x-2)$$

\quad
So using \ref{pr:Proposition_6}, we can determine a particular solution $\bar{u}_3(x,t)$. It can be derived that, supposing that $t>0$,

$$g_{2:3}(x,t)=\ln(t)x, \hspace{0.1cm} R_{f_{3}}(x,t)=\ln(t)\frac{\sinh x}{e^x}, \hspace{0.1cm} G(x,t)=\ln(t)e^{-x}(1-\tanh x).$$

So the particular solution is

$$\bar{u}_3(x,t)= \ln(t)\int_{x_0}^{x}se^{-\frac{s}{2}}\cos\bigg( \frac{\sqrt{3}}{2}s \bigg)ds,$$

which is elementary by \ref{cor:Corollary_1} because it exsits $\nu(x)=e^{-\frac{x}{2}}\cos\big(\frac{2\pi}{3}+ \frac{\sqrt{3}}{2}x \big)$ such that $\nu(x)^{(1+1)}=e^{-\frac{x}{2}}\cos\big( \frac{\sqrt{3}}{2}x \big)$ and $S(x)=0$ thus the solution $\bar{u}_3(x,t)$ satisfies equation \eqref{eq:Equation_31} for $A=1$.

$$\bar{u}_3(x,t)=\ln(t)\bigg[ xe^{-\frac{x}{2}}\cos\bigg(\frac{4\pi}{3} +\frac{\sqrt{3}}{2}x  \bigg)- e^{-\frac{x}{2}}\cos\bigg(\frac{2\pi}{3} +\frac{\sqrt{3}}{2}x  \bigg) +C_3 \bigg]$$

As we mentioned at the start of this subsection \ref{sse:Subsection5_1}, we can add up to particular solution $\bar{u}_3(x,t)$ the solutions, in this case elementary, of the homogeneus version of nested derivative

$$ \partial_{x} \frac{1}{\cosh x}\partial_{x} \frac{\sinh x}{e^{x}}\partial_{x}\frac{e^{\frac{x}{2}}}{\cos\big( \frac{\sqrt{3}}{2}x \big)}\partial_{x}u=0.$$
 
In particular, for the previous example, it's immediate finding an elementary solution $u_1(x,t)$ of

$$\partial_{x}\frac{e^{\frac{x}{2}}}{\cos\big( \frac{\sqrt{3}}{2}x \big)}\partial_{x}u=0,$$

i.e

$$u_1(x,t)=h_1(t)\bigg[e^{-\frac{x}{2}}\cos\bigg( \frac{4\pi}{3} +\frac{\sqrt{3}}{2}x\bigg) +C_1\bigg],$$

\quad
where $C_2$ is an integral constant determined by initial condition $x_0$.

So, the possible solutions of \eqref{eq:Equation_34} are the functions

$$u(x,t)=A_1 h_1(t))\bigg[e^{-\frac{x}{2}}\cos\bigg( \frac{4\pi}{3} +\frac{\sqrt{3}}{2}x\bigg) \bigg]+A_2 h_2(t)+ B \ln(t)\bigg[ xe^{-\frac{x}{2}}\cos\bigg(\frac{4\pi}{3} +\frac{\sqrt{3}}{2}x  \bigg)- e^{-\frac{x}{2}}\cos\bigg(\frac{2\pi}{3} +\frac{\sqrt{3}}{2}x  \bigg)  \bigg],$$

with $A_1$, $A_2$, $B$ opportune constants. But, since we want to cancel the first member, necessarily $\partial_t t \partial_t h_1(t)=\partial_t t \partial_t h_2(t)=0$, then we can consider $h_1(t)=h_2(t)=\ln t$. So we have the following family of admissible solutions to \eqref{eq:Equation_34}.

$$u(x,t)= \ln(t)\biggl\{A_1\bigg[e^{-\frac{x}{2}}\cos\bigg( \frac{4\pi}{3} +\frac{\sqrt{3}}{2}x\bigg) \bigg]+A_2+ B \bigg[ xe^{-\frac{x}{2}}\cos\bigg(\frac{4\pi}{3} +\frac{\sqrt{3}}{2}x  \bigg)- e^{-\frac{x}{2}}\cos\bigg(\frac{2\pi}{3} +\frac{\sqrt{3}}{2}x  \bigg) \bigg]\biggr\}.$$	
\end{example}

\subsection{The homogeneous non-linear case $\partial_{x}D_{f_{l},1}^{(k-1)}P(u)=0$ }\label{sse:Subsection5_2}

\quad

In this subsection, we determine elementary solutions of non-linear nested derivation 

\begin{equation}\label{eq:Equation_35}
\partial_{x}D_{f_{l},1}^{(k-1)}P(u)=0,	
\end{equation}

with $P(u) \in \mathbb{C}[u]$ and $deg_u P(u)>1$.

The first step to find a solution is to set $P(u)=V$ and solve the associated linear equation of \eqref{eq:Equation_35}

$$\partial_{x}D_{f_{l},1}^{(k-1)}V=0.$$

If $V_{j}(x,t)=f(t)\int_{x_0}^{x}\frac{g_{2:j}(s)}{f_1 (s)}ds$ is a  solution of linear equation of \eqref{eq:Equation_35},
since $P(u(x,t))=V(x,t)$, we get that $P(u (x,t))=f(t)\int_{x_0}^{x}\frac{g_{2:j}(s)}{f_1 (s)}ds$.

Now we consider the equation $P(u)-d=0$ with $d$ a complex parameter, i.e., $d \in \mathbb{C}$. 
By the Fundamental Theorem of Algebra, we get $m$ functions $\alpha_{k'} ((c_i)_{i=0}^{l},d)$, if $P(u)=\sum_{i=0}^{l}c_i u^{i}$ with $c_i \in \mathbb{C}$ and $c_l \neq 0$, such that $P(\alpha_{k'} ((c_i)_{i=0}^{l},d))-d=0$ with $1 \leq k' \leq m$ and $m\leq l$.

Replacing $d$ with $V_{j}(x,t)$, we obtain a set of possible solutions of \eqref{eq:Equation_35} 

\begin{equation}\label{eq:Equation_36}
u_{j,k}(x,t)=\alpha_{k'} ((c_i)_{i=0}^{l},V_j(x,t)),	
\end{equation}

for $1 \leq k' \leq m$ with $m \leq l$ and $1 \leq j \leq k-1$ .  

The linearity of $\partial_{x}D_{f_{l},1}^{(k-1)}V=0.$ extends the set of solutions of \eqref{eq:Equation_36} in such way

\begin{equation}\label{eq:Equation_37}
u(x,t)=\alpha_{k'} \bigg((c_i)_{i=0}^{l},\sum_{j=1}^{k-1}A_{j}V_{j}(x,t)\bigg),	
\end{equation}
with $A_{j} \in \mathbb{C}$.

\quad

Among the solutions \eqref{eq:Equation_37}, we want to identify criteria for determining the elementary  ones. In the subsection \ref{sse:Subsection5_1}, we have found criteria for determining elementary solutions of $V_{j}$. 
We still need to establish conditions on the functions $\alpha_{k'}$ when $V(x,t)$ of the type \eqref{eq:Equation_37} is elementary.

\begin{proposition}\label{pr:Proposition_9}
The equation \eqref{eq:Equation_35} 

$$\partial_{x}D_{f_{l},1}^{(k-1)}P(u)=0,$$

has elementary solutions of the type \eqref{eq:Equation_37}, i.e.

$$u(x,t)=\alpha_{k'} \bigg((c_i)_{i=0}^{l},\sum_{s=1}^{k''}A_{j_{s}}V_{j_{s}}(x,t)\bigg),$$

with $k''\leq k-1$, if   $V_{j_{s}}(x,t)$ for $ 1\leq s \leq k''$, solutions of linear equation associated to \eqref{eq:Equation_35}  $\partial_{x}D_{f_{l},1}^{(k-1)}V=0$, satisfy the propositions \ref{pr:Proposition_4} and/or \ref{pr:Proposition_5} and the equation
$P(u)-d=0$ can be resolved by radicals. 
In particular, if $V_1(x,t)=f(t)\int_{x_0}^{x}\frac{1}{f_1 (s)}ds$ is elementary for an opportune $f(t)$ and $P(u)-d=0$ is resolved by radicals then

$$u_{1,k}(x,t)=\alpha_{k'} ((c_i)_{i=0}^{l},V_1 (x,t))$$

are an elementary solutions of  \eqref{eq:Equation_35}.	
\end{proposition}
\begin{proof}

We have already seen (see \ref{pr:Proposition_4} and \ref{pr:Proposition_5}  ) that the function

$$V(x,t)=\sum_{s=1}^{k''}A_{j_{s}}V_{j_{s}}(x,t)$$

is elementary, and since $P(u)-d=0$ is resolved by radicals, the functions $\alpha_{k'} ((c_i)_{i=0}^{l},d)$, such that $P(\alpha_k ((c_i)_{i=0}^{l},d))-d=0$, are algebraic functions.

So the solutions

$$u(x,t)=\alpha_{k'} \bigg((c_i)_{i=0}^{l},\sum_{s=1}^{k''}A_{j_{s}}V_{j_{s}}(x,t)\bigg),$$

are compositions of algebraic functions $\alpha_{k'}$ and elementary functions $V$ and thus $u(x,t)$ are the elementary solutions.	
\end{proof}

\quad

As we saw in the previous section, let us look the example of non-linear 
PDEs in the space-time unidemnsional $(x,t)$ that can be solved using \eqref{eq:Equation_35} .

\quad

\begin{example}\label{ex:Example_4}
	
	Given equation
	
	\begin{equation}\label{eq:Equation_38}
	\partial_{t}^{\nu} \circ P u = \partial_x x \partial_x \frac{\ln x}{\cosh x}\partial_x \tanh x \partial_x \frac{\sinh x}{e^{x}}\partial_x \frac{e^{\frac{x}{2}}}{\cos(\frac{\sqrt{3}}{2}x)}\partial_x (u^{3}-u),	
	\end{equation}
	
	with $ n-1 < \nu < n $ with $n >1$ and $	\partial_{t}^{\nu}$ is the fractional derivative of Caputo (\cite{diethelm2019general})
	
	$$	\partial_{t}^{\nu}f(t)=\frac{1}{\Gamma(n-\nu)}\int_{0}^{t}(t-\tau)^{(n-\nu-1)}f^{(n)}(\tau)d\tau$$
	
	and $P$ is an operator functional polinomial defined in such way
	
	$$P(f(t))=f^{3}(t)-f(t),$$
	
	in particular
	
		$$\partial_{t}^{\nu} \circ P(f(t))=\frac{1}{\Gamma(n-\nu)}\int_{0}^{t}(t-\tau)^{(n-\nu-1)}(f^{3}(\tau)-f(\tau))^{(n)}d\tau.$$ 
		
		\quad
	Let's solve the equation \eqref{eq:Equation_38}	by finding the solutions that make both sides equal to zero. The second member becomes zero if
	
	$$ \partial_x x \partial_x \frac{\ln x}{\cosh x}\partial_x \tanh x \partial_x \frac{\sinh x}{e^{x}}\partial_x \frac{e^{\frac{x}{2}}}{\cos(\frac{\sqrt{3}}{2}x)}\partial_x (u^{3}-u)=0.$$
	
	By \eqref{pr:Proposition_9}, since the equation associated is $u^{3}-u-d=0$ and so it's resolved by radicals because it's a cubic equation and 	$$ \partial_x x \partial_x \frac{\ln x}{\cosh x}\partial_x \tanh x \partial_x \frac{\sinh x}{e^{x}}\partial_x \frac{e^{\frac{x}{2}}}{\cos(\frac{\sqrt{3}}{2}x)}\partial_x V(x,t)=0, $$ admits elementary solutions of the type $$ V(x,t)=A'h_1(t)(e^{-\frac{x}{2}}\cos \bigg(\frac{4 \pi}{3}+\frac{\sqrt{3}}{2}x\bigg)+ B'h_2(t)e^{\frac{x}{2}}\bigg(\cos \bigg(\frac{4 \pi}{3}+\frac{\sqrt{3}}{2}x\bigg)-\cos \bigg(\frac{2 \pi}{3}+\frac{\sqrt{3}}{2}x\bigg) + \cos \bigg(\frac{\sqrt{3}}{2}x\bigg)   \bigg) + C'h_3 (t)     $$
	(see example \ref{ex:Example_2} ), the equation \eqref{eq:Equation_38} has the elementary solutions of the type \eqref{eq:Equation_37}

	$$u(x,t)=\alpha_{k'} \bigg(0,-1,0,1,A'V_1(x,t)+B'V_2(x,t)+C'h_3 (t) \bigg),$$
	
	with $ 1\leq k' \leq m$ and $m \leq 3$ and 
	
	$$V_1 (x,t)=A'h_1(t)(e^{-\frac{x}{2}}\cos \bigg(\frac{4 \pi}{3}+\frac{\sqrt{3}}{2}x\bigg),$$
	
	$$V_3 (x,t)=h_2(t)e^{\frac{x}{2}}\bigg(\cos \bigg(\frac{4 \pi}{3}+\frac{\sqrt{3}}{2}x\bigg)-\cos \bigg(\frac{2 \pi}{3}+\frac{\sqrt{3}}{2}x\bigg) + \cos \bigg(\frac{\sqrt{3}}{2}x\bigg)   \bigg).$$	
	
	We can determine explicity the $\alpha_k$ functions with Del Ferro-Fontana's formula of depressed cubic (\cite{gavagna2012soluzione}) which we recall here.
\quad

If $z^{3}+pz+q=0$ the solutions are 

$$z_1=x_1+x_2$$

$$z_2=\zeta_{3}x_1+ \zeta_{3}^{2}x_2$$

$$z_3=\zeta_{3}^{2}x_1+ \zeta_{3}x_2,$$ 

with 

$$x_1=\sqrt[3]{-\frac{q}{2}+\sqrt{\frac{q^{2}}{4}+\frac{p^{3}}{27}}}$$

$$x_2=\sqrt[3]{-\frac{q}{2}-\sqrt{\frac{q^{2}}{4}+\frac{p^{3}}{27}}},$$  	
and $\zeta_{3}=e^{\frac{2i\pi}{3}}$ the primitive root of unity of order three.	

In this case $q=-V(x,t)$ and $p=-1$ so we get

$$x_1(x,t)=\sqrt[3]{\frac{V(x,t)}{2}+\sqrt{\frac{(V(x,t))^{2}}{4}-\frac{1}{27}}}$$

$$x_2(x,t)=\sqrt[3]{\frac{V(x,t)}{2}-\sqrt{\frac{(V(x,t))^{2}}{4}-\frac{1}{27}}}	.$$

So the potentials solutions of \eqref{eq:Equation_38} are

$$z_l(x,t)=\zeta_{3}^{l-1}x_1 (x,t)+ \zeta_{3}^{3-l+1}x_2 (x,t)$$

with $ 1 \leq l \leq 3$.

We still need to determine the time-coefficients of $V(x,t)$ i.e $h_l(t)$ for $1 \leq l \leq 3$ such that the member with operator non-linear $\partial_{t}^{\nu} \circ P(f(t))$ is zero.

\quad

$$	\partial_{t}^{\nu} \circ P z_l (x,t)=\frac{1}{\Gamma(n-\nu)}\int_{0}^{t}(t-\tau)^{(n-\nu-1)}\partial_{\tau}^{(n)}(z_{l}^{3}(x,\tau)-z_{l}(x,\tau))d\tau=\frac{1}{\Gamma(n-\nu)}\int_{0}^{t}(t-\tau)^{(n-\nu-1)}\partial_{\tau}^{(n)}V(x,\tau)d \tau=
 $$ 
 
 $$=\frac{1}{\Gamma(n-\nu)}\int_{0}^{t}(t-\tau)^{(n-\nu-1)}\partial_{\tau}^{(n)}(A'V_1(x,\tau)+B'V_2(x,\tau)+C'h_3 (\tau))d \tau=$$
 
 $$\hspace{-2cm}=A'\frac{1}{\Gamma(n-\nu)}\int_{0}^{t}(t-\tau)^{(n-\nu-1)}\partial_{\tau}^{(n)}V_1(x,\tau)d\tau+B'\frac{1}{\Gamma(n-\nu)}\int_{0}^{t}(t-\tau)^{(n-\nu-1)}\partial_{\tau}^{(n)}V_2(x,\tau)d\tau+C'\frac{1}{\Gamma(n-\nu)}\int_{0}^{t}(t-\tau)^{(n-\nu-1)}\partial_{\tau}^{(n)}h_3 (\tau)d\tau=$$
 
 $$\hspace{-2cm} =A'e^{-\frac{x}{2}}\cos \bigg(\frac{4 \pi}{3}+\frac{\sqrt{3}}{2}x\bigg)\frac{1}{\Gamma(n-\nu)}\int_{0}^{t}(t-\tau)^{(n-\nu-1)}\partial_{\tau}^{(n)}h_1 (\tau)d\tau$$
 
 $$\hspace{-2cm}+B'e^{\frac{x}{2}}\bigg(\cos \bigg(\frac{4 \pi}{3}+\frac{\sqrt{3}}{2}x\bigg)-\cos \bigg(\frac{2 \pi}{3}+\frac{\sqrt{3}}{2}x\bigg) + \cos \bigg(\frac{\sqrt{3}}{2}x\bigg)   \bigg)\frac{1}{\Gamma(n-\nu)}\int_{0}^{t}(t-\tau)^{(n-\nu-1)}\partial_{\tau}^{(n)}h_2 (\tau)d\tau$$
 $$+C'\frac{1}{\Gamma(n-\nu)}\int_{0}^{t}(t-\tau)^{(n-\nu-1)}\partial_{\tau}^{(n)}h_3 (\tau)d\tau.$$
 
Possible solutions are found by solving the following equations

$$\partial_{\tau}^{(n)}h_l (\tau)=0$$

for $1 \leq l \leq 3$.

Therefore 

\begin{equation}\label{eq:Equation_39}
h_l(t)=\sum_{j=0}^{m_l}C_{j,l} t^{j},		
\end{equation}

for opportune constant $C_{j,l}$ such that $C_{j',l}\neq 0$ for an opportune $j'>1$ and $1 \leq m_l <n$ for $1\leq l \leq 3$.

We can therefore summarize the complete set of valid solutions of equation \eqref{eq:Equation_38} .

\begin{equation}\label{eq:Equation_40}	
z_l(x,t)=\zeta_{3}^{l-1}x_1 (x,t)+ \zeta_{3}^{3-l+1}x_2 (x,t),	
\end{equation}

with $x_1(x,t)=\sqrt[3]{\frac{V(x,t)}{2}+\sqrt{\frac{(V(x,t))^{2}}{4}-\frac{1}{27}}}$ and $x_2(x,t)=\sqrt[3]{\frac{V(x,t)}{2}-\sqrt{\frac{(V(x,t))^{2}}{4}-\frac{1}{27}}}$ and

$$V(x,t)=\sum_{j=0}^{m_1}C_{j,1} t^{j}e^{-\frac{x}{2}}\cos \bigg(\frac{4 \pi}{3}+\frac{\sqrt{3}}{2}x\bigg)+\sum_{j=0}^{m_2}C_{j,2} t^{j} e^{\frac{x}{2}}\bigg(\cos \bigg(\frac{4 \pi}{3}+\frac{\sqrt{3}}{2}x\bigg)-\cos \bigg(\frac{2 \pi}{3}+\frac{\sqrt{3}}{2}x\bigg) + \cos \bigg(\frac{\sqrt{3}}{2}x\bigg)   \bigg) + \sum_{j=0}^{m_3}C_{j,3} t^{j}$$

for opportune polynomials $\sum_{j=0}^{m_l}C_{j,l} t^{j}$  defined in \eqref{eq:Equation_39}, for $1 \leq l \leq 3$.

\end{example}

\subsection{The non-homogeneus non-linear case $\partial_{x}D_{f_{l},1}^{(k-1)}P(u)=g(x,t)$ }\label{sse:Subsection5_3}

\quad

In this subsection we study the equation

\begin{equation}\label{eq:Equation_41}
\partial_{x}D_{f_{l},1}^{(k-1)}P(u)=g(x,t).
\end{equation}
\quad
As in the previous section, $deg_{u}P(u)>1$, and thus the first step consists of analyzing the linear equation associated with \eqref{eq:Equation_41}, posing $V(x,t)=P(x,t)$,

$$\partial_{x}D_{f_{l},1}^{(k-1)}V=g(x,t).$$

Which is the equation of the type \eqref{eq:Equation_33}, and so the solutions are, using notations of \ref{sse:Subsection5_1}

$$V(x,t)=\sum_{j=0}^{k-1}A_j V_j(x,t)+ B\bar{V}_{k-1}(x,t),$$

where we recall that $V_j(x,t)$ are solutions in the homogeneous case

$$\partial_{x}D_{f_{l},1}^{(k-1)}V=0$$ and $\bar{V}_{k-1}(x,t)$ are particular solutions of the non-homogeneous case, relating to the case \eqref{eq:Equation_33}.

And again, as the previous section, the solution $u(x,t)$ is derived from the equation $P(u)-V(x,t)=0$, and it follows that

$$u(x,t)=\alpha_{k'} \bigg((c_i)_{i=0}^{l},\sum_{j=0}^{k-1}A_{j}V_{j}(x,t)+ B\bar{V}_{k-1}(x,t)\bigg),$$

where $\alpha_{k'}$ are the functions that solve, as in the previous section, the equation $P(u)-V(x,t)=0$.

Following these remarks, we can establish a criterion for determining the existence of elementary solutions of \eqref{eq:Equation_41} similar to that used in the homogeneous case.

\begin{proposition}\label{pr:Proposition_10}
	The equation \eqref{eq:Equation_41} 
	
	$$\partial_{x}D_{f_{l},1}^{(k-1)}P(u)=g(x,t),$$
	
	has elementary solutions of the type   
	
	$$u(x,t)=\alpha_{k'} \bigg((c_i)_{i=0}^{l},\sum_{s=1}^{k''}A_{j_{s}}V_{j_{s}}(x,t)+B\bar{V}_{k-1}(x,t)\bigg),$$
	
	with $k''\leq k-1$, if   $V_{j_{s}}(x,t)$ for $ 1\leq s \leq k''$, solutions of the linear equation associated with \eqref{eq:Equation_35}
	$$ \partial_{x}D_{f_{l},1}^{(k-1)}V=0$$ satisfy the propositions \ref{pr:Proposition_4} and/or \ref{pr:Proposition_5} and $\bar{V}_{k-1}(x,t)$, which is a particular solution associated with equation \eqref{eq:Equation_33} i.e
	
	$$\partial_{x}D_{f_{l},1}^{(k-1)}V=g(x,t)$$ ,  satisfy the propositions \ref{pr:Proposition_7} and/or \ref{pr:Proposition_8} and the equation
	$P(u)-d=0$ can be resolved by radicals. 	
\end{proposition}

\quad
Let's look at an example of where to apply the proposition \ref{pr:Proposition_10}.

\begin{example}\label{ex:Example_5}
Given non-linear PDE

\begin{equation}\label{eq:Equation_42}
\partial_{t}^{\nu} \circ P u=\partial_{x} \frac{1}{\cosh x}\partial_{x} \frac{\sinh x}{e^{x}}\partial_{x}\frac{e^{\frac{x}{2}}}{\cos\big( \frac{\sqrt{3}}{2}x \big)}\partial_{x}(u^{4}-u)+P_{n-1}(t)
e^{-x}(\tanh^{2} x+\tanh x-2),	
\end{equation}	
	with $ n-1 < \nu < n $ with $n >1$ and $	\partial_{t}^{\nu}$ is the fractional derivative of Caputo (see example \ref{ex:Example_4}), and $P$ is an operator functional polynomial defined in such a way
	
	$$P(f(t))=f^{4}(t)-f(t).$$
	
	Whereas $P_{n-1}(t)$ is a fixed polynomial function dependig on the time variable i.e
	
	$$P_{n-1}(t)=\sum_{j=0}^{n-1}B_j t^{j},$$
	with $B_{n-1}\neq 0.$
	
	\quad
	As in example \ref{ex:Example_4}, we want to solve the equation looking for solutions that cancel out two members. Cancelling the second member implies
	
	$$\partial_{x} \frac{1}{\cosh x}\partial_{x} \frac{\sinh x}{e^{x}}\partial_{x}\frac{e^{\frac{x}{2}}}{\cos\big( \frac{\sqrt{3}}{2}x \big)}\partial_{x}(u^{4}-u)=-P_{n-1}(t)
	e^{-x}(\tanh^{2} x+\tanh x-2).$$	
	First of all, we must resolve the linear equation associated
	
	$$\partial_{x} \frac{1}{\cosh x}\partial_{x} \frac{\sinh x}{e^{x}}\partial_{x}\frac{e^{\frac{x}{2}}}{\cos\big( \frac{\sqrt{3}}{2}x \big)}\partial_{x}(V(x,t))=-P_{n-1}(t)
e^{-x}(\tanh^{2} x+\tanh x-2),$$	
	 with $V(x,t)=u^{4}(x,t)-u(x,t)$.
	 
	 This equation can be resolved in a similar way to example \ref{ex:Example_3}, so the possible solutions of the equation above are
	 
	 $$ V(x,t)=A_1 h_1(t))\bigg[e^{-\frac{x}{2}}\cos\bigg( \frac{4\pi}{3} +\frac{\sqrt{3}}{2}x\bigg) \bigg]+A_2 h_2(t)+  P_{n-1}(t)\bigg[ xe^{-\frac{x}{2}}\cos\bigg(\frac{4\pi}{3} +\frac{\sqrt{3}}{2}x  \bigg)- e^{-\frac{x}{2}}\cos\bigg(\frac{2\pi}{3} +\frac{\sqrt{3}}{2}x  \bigg)  \bigg].$$
		
	The second step consists of resolving the quartic equation
	
	$$ u^{4}(x,t)-u(x,t)-V(x,t)=0,$$
	
	since it's a quartic equation, it's solvable by radicals.
	
	Now we recall Ferrari's Formulas (\cite{gavagna2012soluzione}) for depressed quartic equation
	
	$$z^{4}+pz^{2}+qz+r=0,$$
	
	in particular, we need the case of $p=0$, i.e
	
	$$z^{4}+qz+r=0.$$
	
	Such formulas in this last case are
	
	$$z_{1,2}=\frac{1}{2}\Bigg( \sqrt{z_0} \pm \sqrt{-2z_0+ \frac{2q}{\sqrt{2z_0}}} \Bigg)$$
	
		$$z_{3,4}=\frac{1}{2}\Bigg(-\sqrt{z_0} \pm \sqrt{-2z_0- \frac{2q}{\sqrt{2z_0}}} \Bigg),$$
		
		where $z_0$ is a solution of the cubic equation $z_{0}^{3}-rz_{0}-\frac{q^{2}}{8}$, so
		
		$$z_{0}=\sqrt[3]{-\frac{q^{2}}{16}+\sqrt{\frac{q^{4}}{256}-\frac{r^{3}}{27}}}+\sqrt[3]{-\frac{q^{2}}{16}-\sqrt{\frac{q^{4}}{256}-\frac{r^{3}}{27}}}. $$
		
	Let's apply these formulas to equation  $ u^{4}(x,t)-u(x,t)-V(x,t)=0$	and we will obtain possible solutions of \eqref{eq:Equation_42}

	\begin{equation}\label{eq:Equation_43}
	z_{1,2}(x,t)=\frac{1}{2}\Bigg( \sqrt{z_0(x,t)} \pm \sqrt{-2z_0(x,t)-\frac{2}{\sqrt{2z_0(x,t)}}} \Bigg)	
	\end{equation}		
		$$z_{3,4}(x,t)=\frac{1}{2}\Bigg(-\sqrt{z_0(x,t)} \pm \sqrt{-2z_0(x,t)+\frac{2}{\sqrt{2z_0}}} \Bigg),$$
		
		where
		\begin{equation}\label{eq:Equation_44}
		z_{0}(x,t)=\sqrt[3]{-\frac{1}{16}+\sqrt{\frac{1}{256}+\frac{(V(x,t))^{3}}{27}}}+\sqrt[3]{-\frac{1}{16}-\sqrt{\frac{1}{256}+\frac{(V(x,t))^{3}}{27}}}.
		\end{equation}
		\quad
		To fully determine these types of solutions, we must calculate the coefficients $h_i(t)$ for $i=1,2$ such that the left-hand side of $\eqref{eq:Equation_42}$ is zero. Since  the left-hand side is operator $\partial_{t}^{\nu} \circ P $, the calculation is similar to example \ref{ex:Example_4}, and it can be derived
		
		that a solution is
		
		$$h_l(t)=\sum_{j=0}^{m_l}C_{j,l} t^{j},$$
		
		for opportune constant $C_{j,l}$ such that $C_{j',l}\neq 0$ for an opportune $j'>1$ and $1 \leq m_l <n$ for $1\leq l \leq 2.$
		
	Furthermore, one can see that $\partial_{t}^{\nu} \circ P P_{n-1}(t)=0$. 
	\quad
	So we finally obtain the possible complete solutions for \eqref{eq:Equation_42}	
	 
	\begin{equation}\label{eq:Equation_45}
	z_{k'}(x,t)=\alpha_{k'}\bigg(0,-1,0,0,1,A'V_1(x,t)+B'\bar{V}_3(x,t)+C'h_2 (t) \bigg),	
	\end{equation}
	with $1 \leq k' \leq 4$ and $\alpha_{k'}$ are the functions \eqref{eq:Equation_43} and 
	
	$$V_1(x,t)=\sum_{j=0}^{m_1}C_{j,1} t^{j}\bigg[e^{-\frac{x}{2}}\cos\bigg( \frac{4\pi}{3} +\frac{\sqrt{3}}{2}x\bigg) \bigg]  $$
	
	$$h_2(t)=\sum_{j=0}^{m_2}C_{j,2} t^{j}$$
	
	$$\bar{V}_3(x,t)= \sum_{j=0}^{n-1}B_j t^{j} \bigg[ xe^{-\frac{x}{2}}\cos\bigg(\frac{4\pi}{3} +\frac{\sqrt{3}}{2}x  \bigg)- e^{-\frac{x}{2}}\cos\bigg(\frac{2\pi}{3} +\frac{\sqrt{3}}{2}x  \bigg)  \bigg] . $$
	
	The latter function is the particular solution for \eqref{eq:Equation_42}, and recall that $B_{n-1}\neq 0$.
\end{example}

\quad

\section{Criteria whether elementary solutions exist for non-linear PDE $\mathcal{O}_t u=\mathcal{O}_x u $ using the equations $\partial_{x}D_{f_{l},1}^{(k-1)}P(u)=0$, $\partial_{x}D_{f_{l},1}^{(k-1)}P(u)=g(x,t)$  }\label{se:Section_6}

\quad

Let's return to the PDEs introduced in the section \ref{Se:Section_1}

$$\mathcal{O}_t u=\mathcal{O}_x u .$$

We recall that we got two strategies to solve them 

	$$ 1)\mathcal{O}_t u=0 , \mathcal{O}_x u=0     $$
$$ 2)\mathcal{O}_x u \in W=<g_1(x),...,g_s(x)>, \hspace{0.1cm} if\hspace{0.1cm} u(x,t) \in W \hspace{0.1cm}\mathcal{O}_t h_i(t)=\lambda(t) h_i(t)\hspace{0.1cm} 1\leq i \leq s,$$

if $u(x,t)= \sum_{i=1}^{s}h_i(t)g_i(x)$.

As we saw in section \ref{Se:Section_1}, the operator $\mathcal{O}_x u$ in any case 1) and 2) can be viewed as in \eqref{eq:Equation_1}

$$\mathcal{O}_x u=\mathcal{\overline{O}}_x u + \lambda(t) u.$$

So we must concentrate on

$$\mathcal{\overline{O}}_x u=0.$$ 

If $\mathcal{\overline{O}}_x u$ is endowed with $nested$ $partition$ (see \eqref{eq:Equation_15}, \eqref{eq:Equation_16}, \eqref{eq:Equation_17}) i.e 

$$\mathcal{\overline{O}}_x u=\sum_{j'=1}^{t'}a_{j'}(x,t)\prod_{j''=0}^{l''}\big(\partial_{x}^{(j'')}u\big)^{i_{j''}}	\partial_{x}D_{f_{l},1}^{(k_{j'}-1)}P(u).                                            $$

Then there's an index $\tilde{j}$ such that $\tilde{k}=k_{\tilde{j}}=min_{1\leq j' \leq t'} k_{j'}$ and

$$\partial_{x}D_{f_{l},1}^{(\tilde{k}-1)}P(u)=0 \rightarrow \mathcal{\overline{O}}_x u=0.$$

Then the solutions of $\partial_{x}D_{f_{l},1}^{(\tilde{k}-1)}P(u)=0 $ are solutions of $\mathcal{\overline{O}}_x u=0$, and from this fact we can determine if there are elementary solutions of $\mathcal{\overline{O}}_x u=0$.

\quad

We also want to find the operators such that 

$$\partial_{x}D_{f_{l},1}^{(\tilde{k}-1)}P(u)=g(x,t) \rightarrow \mathcal{\overline{O}}_x u=0.$$

\quad

And this is possible if it exists an integer $\tilde{k}$ such that 

\begin{equation}\label{eq:Equation_46}
\mathcal{\overline{O}}_x u=\sum_{j'=1,k_{j'}-1\geq \tilde{k}}^{t'}a_{j'}(x,t)\prod_{j''=0}^{l''}\big(\partial_{x}^{(j'')}u\big)^{i_{j''}}	\bigg(\partial_{x}D_{f_{l},1}^{(k_{j'}-1)}P(u)-\partial_{x}D_{f_{l},\tilde{k}}^{(k_{j'}-1)}g(x,t)\bigg)+
\end{equation}

$$\sum_{j'=1, k_{j'}=\tilde{k}}^{t'}a_{j'}(x,t)\prod_{j''=0}^{l''}\big(\partial_{x}^{(j'')}u\big)^{i_{j''}}	\bigg(\partial_{x}D_{f_{l},1}^{(\tilde{k}-1)}P(u)-g(x,t)\bigg). $$

\quad
In this case we'll say that $\mathcal{\overline{O}}_x u$ has got a $nested $ $partition$ non-homogeneous.

To summarize, we have the following properties

\begin{proposition}\label{pr:Proposition_11}
Given 

$$\mathcal{\overline{O}}_x u=0,$$

where  $\mathcal{\overline{O}}_x u \in \mathcal{C}^{\infty}(U;\mathbb{C})\big[\partial^{(k)}_{x}u \big]_{k=0}^{l}$. If $\mathcal{\overline{O}}_x u$ has a $nested$ $partition$ and $\mathcal{\overline{O}}_x u=0$ if $\partial_{x}D_{f_{l},1}^{(k-1)}P(u)=0$, for opportune functions $f_{l}$, and polinomial $P(u) \in \mathbb{C}[u]$.
Then the solutions of $\mathcal{\overline{O}}_x u=0$ are elementary solutions if $\partial_{x}D_{f_{l},1}^{(k-1)}P(u)=0$ satisfies the proposition \ref{pr:Proposition_9}.
\end{proposition}

\quad

\begin{proposition}\label{pr:Proposition_12}
	Given 
	
	$$\mathcal{\overline{O}}_x u=0,$$
	
	where  $\mathcal{\overline{O}}_x u \in \mathcal{C}^{\infty}(U;\mathbb{C})\big[\partial^{(k)}_{x}u \big]_{k=0}^{l}$. If $\mathcal{\overline{O}}_x u$ has a $nested$ $partiton$ non-homogeneous and $\mathcal{\overline{O}}_x u=0$ if $\partial_{x}D_{f_{l},1}^{(k-1)}P(u)=g(x,t)$, for  opportune functions $f_{l}$,$g(x,t)$ and polinomial $P(u) \in \mathbb{C}[u]$.
	Then the solutions of $\mathcal{\overline{O}}_x u=0$ are solutions elementaries if $\partial_{x}D_{f_{l},1}^{(k-1)}P(u)=g(x,t)$ satisfies the proposition \ref{pr:Proposition_10}.
\end{proposition}
\quad
Now we see some examples of PDE non-linear $\mathcal{O}_t u=\mathcal{O}_x u$ where we can apply the propositions \ref{pr:Proposition_11}, \ref{pr:Proposition_12}.

\begin{example}\label{ex:Example_6}
	Given the equation
	
	\begin{equation}\label{eq:Equation_47}
		\partial_{t}^{\nu} \circ P u=3\bigg(1-\frac{1}{\tanh^{2}x}\bigg)u^{4}\partial_{x}u+\bigg(\frac{1}{\tanh^{2}x}+\frac{2}{\tanh x}-1\bigg)u^{2}\partial_{x}u+\frac{6}{\tanh x}u^{3}(\partial_{x}u)^{2}+\frac{3}{\tanh x}u^{4}\partial_{x}^{(2)}u
	\end{equation}
	
	$$ +6u^{2}(\partial_{x}u)^{3}+18u^{3}\partial_{x}u\partial_{x}^{(2)}u+3u^{4}\partial_{x}^{(3)}u-u^{2}\partial_{x}^{(3)}+3\bigg(1-\frac{1}{\tanh^{2}x}\bigg)u^{3}\partial_{x}u+\bigg(\frac{1}{\tanh^{2}x}-\frac{1}{\tanh x}-1\bigg)u\partial_{x}u$$
	
	$$+\frac{6}{\tanh x}u^{2}(\partial_{x}u)^{2}+\frac{3}{\tanh x}u^{3}\partial_{x}^{(2)}u+6u(\partial_{x}u)^{3}+18u^{2}\partial_{x}u\partial_{x}^{(2)}u+3u^{3}\partial_{x}^{(3)}u-u\partial_{x}^{(3)}u-\frac{1}{\tanh x}\partial_{x}u $$
	
	$$ +6u(\partial_{x}u)^{2}+3u^{2}\partial_{x}^{(2)}u-\partial_{x}^{(2)}u,$$
	
	where $\partial_{t}^{\nu} \circ P$ is a temporal non-linear operator defined in the example \ref{ex:Example_4}.
	\quad

	While the operator $\mathcal{O}_x$ is endowed with $nested$ $partition$, i.e., $\mathcal{O}_x$ satisfies the equation \eqref{eq:Equation_14}. This it can verify. Infact, we can remark that the maximum degree of derivation is 3, so the nested derivation grade is 3, and there will exist two functions $f_1$ and $f_2$ that define the nested derivation. Then we must verify which type of polynomial defines this nested derivation operator by the Faa' di Bruno's Formula. In this case we need just to look that there are coefficients of  $\mathcal{O}_x$ which occur repetitvely, i.e., $3\bigg(1-\frac{1}{\tanh^{2}x}\bigg)$, $\frac{6}{\tanh x}$, $\frac{3}{\tanh x}$. Now, we find which Faa' di Bruno's Formula is opportune and then we obtain that
	 $$\partial_{t}^{\nu} \circ P u=\bigg(1-\frac{1}{\tanh^{2}x}\bigg)(u^{2}+u)\partial_{x}Q(u)+\frac{1}{\tanh x}(u^{2}+u)\partial_{x}^{(2)}Q(u)+ (u^{2}+u)\partial_{x}^{(3)}Q(u)+\frac{1}{\tanh x} \partial_{x}Q(u)+\partial_{x}^{(2)}Q(u),$$
	 
	 where $Q(u)=u^{3}-u$.
	To conclude, we need to see how $\mathcal{O}_x$ is an algebraic combination of nested derivatives. To do this, we apply the proposition \ref{pr:Proposition_1} to $\mathcal{O}_x$. In order to obtain $f_1$ and $f_2$ we just remark that $\tanh x=\frac{\sinh x}{\cosh x}$ and $1=\frac{\sinh x}{\sinh x}$, and so we discover that $f_1=\sinh x$ , $f_2=\frac{1}{\sinh x}$.
	Thus we get
	
	$$\partial_{t}^{\nu} \circ P u=u^{2}\partial_x \frac{1}{\sinh x}\partial_{x} \sinh x \partial_{x} (u^{3}-u)+ u \partial_x \frac{1}{\sinh x}\partial_{x} \sinh x \partial_{x} (u^{3}-u)+\frac{1}{\sinh x}\partial_{x} \sinh x \partial_{x} (u^{3}-u).$$
	
	This equation can be rewritten in a compact way
	
	$$\partial_{t}^{\nu} \circ P u= u^{2}\partial_{x}\Delta_H (u^{3}-u)+  u\partial_{x} \Delta_H (u^{3}-u)+\Delta_H (u^{3}-u),$$
	
	where $\Delta_H=\frac{1}{\sinh x}\partial_{x} \sinh x $ is Hyperbolic Laplacian (\cite{lao2007hyperbolic}).
	
	 In any case, applying the proposition \ref{pr:Proposition_11} to \eqref{eq:Equation_47} , we obtain
	 
	 $$\mathcal{O}_{x}=0 \hspace{0.2 cm}if\hspace{0.2 cm} \partial_{x} \sinh x\partial_{x}(u^{3}-u)=0.$$ 
	 
	 Since $\frac{1}{\sinh x}$ has an elementary antiderivative function, we obtain elementary solutions (we recall that $u^{3}-u$ is solvable by radicals) that we can determine explicitly.
	 
	 $$\frac{1}{\sinh x}\partial_{x} \sinh x\partial_{x}(u^{3}-u)=0 \rightarrow u^{3}-u=h_1(t)\ln\Big(\tanh \Big(\frac{x}{2}\Big)\Big)+h_2(t).$$(for this type of solutions, see \cite{garra2023note}).
	 We use the formulas applied in the example \ref{ex:Example_4}, and we yield
	 $$ \alpha_{k'}(0,-1,0,1,h_1(t)\ln\big(\tanh \big(\frac{x}{2}\big)\big)+h_2(t) )=\zeta_{3}^{k'}v_1(x,t)+\zeta_{3}^{3-k'}v_2(x,t),$$
	 
	 with $ 1 \leq k' \leq 3$ and 
	 
	 $$v_1(x,t)=\sqrt[3]{\frac{h_1(t)\ln\big(\tanh \big(\frac{x}{2}\big)\big)+h_2(t)}{2}+\sqrt{\frac{(h_1(t)\ln\big(\tanh \big(\frac{x}{2}\big)\big)+h_2(t))^{2}}{4}-\frac{1}{27}}}$$
	 
	 $$ v_2 (x,t)=\sqrt[3]{\frac{h_1(t)\ln\big(\tanh \big(\frac{x}{2}\big)\big)+h_2(t)}{2}-\sqrt{\frac{(h_1(t)\ln\big(\tanh \big(\frac{x}{2}\big)\big)+h_2(t))^{2}}{4}-\frac{1}{27}}}	.$$ 
	 
	 Now, to conclude this example and find definitively the solutions of \eqref{eq:Equation_47}, we must calculate $h_1(t)$, $h_2(t)$ by the equation $\partial_{t}^{\nu} \circ P u=0$.
	 
	The computation is similar to \ref{ex:Example_5}, and the solutions are 
	\begin{equation}\label{eq:Equation_48}
	h_1(t)=\sum_{j=0}^{m_1}C_{j,1} t^{j}	
	\end{equation} 
	 $$	h_2(t)=\sum_{j=0}^{m_2}C_{j,2} t^{j},                                $$
	 
	 where $n>m_1>1$, $n>m_2>1$ such that $n>1$ and $n-1 < \nu <n$, i.e., $h_1(t), h_2(t)$ are non-trivial solutions.
	 
	 So the possible solutions are \eqref{eq:Equation_47}
	 
	 \begin{equation}\label{eq:Equation_49}
	 \alpha_{k'}(0,-1,0,1,h_1(t)\ln\big(\tanh \big(\frac{x}{2}\big)\big)+h_2(t) )=\zeta_{3}^{k'}v_1(x,t)+\zeta_{3}^{3-k'}v_2(x,t),	
	 \end{equation}
	 
	 with $ 1 \leq k' \leq 3$ and
	 
	 $$v_1(x,t)=\sqrt[3]{\frac{\sum_{j=0}^{m_1}C_{j,1} t^{j}\ln\big(\tanh \big(\frac{x}{2}\big)\big)+\sum_{j=0}^{m_2}C_{j,2} t^{j}}{2}+\sqrt{\frac{(\sum_{j=0}^{m_1}C_{j,1} t^{j}\ln\big(\tanh \big(\frac{x}{2}\big)\big)+\sum_{j=0}^{m_2}C_{j,2} t^{j})^{2}}{4}-\frac{1}{27}}}$$

	 $$ v_2 (x,t)=\sqrt[3]{\frac{\sum_{j=0}^{m_1}C_{j,1} t^{j}\\ln\big(\tanh \big(\frac{x}{2}\big)\big)+\sum_{j=0}^{m_2}C_{j,2} t^{j}}{2}-\sqrt{\frac{(\sum_{j=0}^{m_1}C_{j,1} t^{j}\ln\big(\tanh \big(\frac{x}{2}\big)\big)+\sum_{j=0}^{m_2}C_{j,2} t^{j})^{2}}{4}-\frac{1}{27}}}	,$$
	 
    with $h_1(t)$, $h_2(t)$ defined in \eqref{eq:Equation_48}.

\end{example}

\begin{example}\label{ex:Example_7}
Given the following non-linear PDE, which is the modified version of \eqref{eq:Equation_47} in this way

\begin{equation}\label{eq:Equation_50}
	\partial_{t}^{\nu} \circ P u=3\bigg(1-\frac{1}{\tanh^{2}x}\bigg)u^{4}\partial_{x}u+\bigg(\frac{1}{\tanh^{2}x}+\frac{2}{\tanh x}-1\bigg)u^{2}\partial_{x}u+\frac{6}{\tanh x}u^{3}(\partial_{x}u)^{2}+\frac{3}{\tanh x}u^{4}\partial_{x}^{(2)}u
\end{equation}

$$ +6u^{2}(\partial_{x}u)^{3}+18u^{3}\partial_{x}u\partial_{x}^{(2)}u+3u^{4}\partial_{x}^{(3)}u-u^{2}\partial_{x}^{(3)}+3\bigg(1-\frac{1}{\tanh^{2}x}\bigg)u^{3}\partial_{x}u+\bigg(\frac{1}{\tanh^{2}x}-\frac{1}{\tanh x}-1\bigg)u\partial_{x}u$$

$$+\frac{6}{\tanh x}u^{2}(\partial_{x}u)^{2}+\frac{3}{\tanh x}u^{3}\partial_{x}^{(2)}u+6u(\partial_{x}u)^{3}+18u^{2}\partial_{x}u\partial_{x}^{(2)}u+3u^{3}\partial_{x}^{(3)}u-u\partial_{x}^{(3)}u-\frac{1}{\tanh x}\partial_{x}u $$

$$ +6u(\partial_{x}u)^{2}+3u^{2}\partial_{x}^{(2)}u-\partial_{x}^{(2)}u-$$

$$\sum_{j=0}^{n_3}C_j t^{j}e^{-\frac{x}{2}}\bigg\{\cos \bigg(\frac{2\pi}{3}+\frac{\sqrt{3}}{2}x  \bigg)+ \frac{1}{\tanh x}\bigg[\cos \bigg(\frac{\sqrt{3}}{2}x  \bigg)+  \frac{\tanh^{2}x-1}{\tanh x}\cos\bigg(\frac{4\pi}{3}+\frac{\sqrt{3}}{2}x  \bigg)\bigg]\bigg \}u^{2}-$$

$$\sum_{j=0}^{n_3}C_j t^{j}e^{-\frac{x}{2}}\bigg\{\cos \bigg(\frac{2\pi}{3}+\frac{\sqrt{3}}{2}x  \bigg)+ \frac{1}{\tanh x}\bigg[\cos \bigg(\frac{\sqrt{3}}{2}x  \bigg)+  \frac{\tanh^{2}x-1}{\tanh x}\cos\bigg(\frac{4\pi}{3}+\frac{\sqrt{3}}{2}x  \bigg)\bigg]\bigg \}u-$$
$$ \sum_{j=0}^{n_3}C_j t^{j}e^{-\frac{x}{2}}\bigg[\frac{1}{\tanh x} \cos\bigg(\frac{4\pi}{3}+\frac{\sqrt{3}}{2}x  \bigg)+\cos \bigg(\frac{\sqrt{3}}{2}x  \bigg) \bigg] ,                         $$

$ 1 \leq n_3 \leq n$ with $n>1$ such that $n-1< \nu <n$ and $C_{n_3}\neq 0$.
The change of equation \eqref{eq:Equation_47} consists of adding the last three lines. It is verified that $\mathcal{\overline{O}}_{x}$ of \eqref{eq:Equation_50} satisfies \eqref{eq:Equation_46} for an operator with a maximum degree of nested derivation $\partial_{x}\frac{1}{\sinh x}\partial_{x}\sinh x \partial_{x}Q(u)$ (see $Q(u)$ of \ref{ex:Example_6}) and the function $$g(x,t)=\sum_{j=0}^{n_3}C_j t^{j}e^{-\frac{x}{2}}\bigg[\frac{1}{\tanh x} \cos\bigg(\frac{4\pi}{3}+\frac{\sqrt{3}}{2}x  \bigg)+\cos \bigg(\frac{\sqrt{3}}{2}x  \bigg) \bigg] .$$

And so

$$ \frac{1}{\sinh x}\partial_{x}\sinh x \partial_{x}Q(u)=\sum_{j=0}^{n_3}C_j t^{j}e^{-\frac{x}{2}}\bigg[\frac{1}{\tanh x} \cos\bigg(\frac{4\pi}{3}+\frac{\sqrt{3}}{2}x  \bigg)+\cos \bigg(\frac{\sqrt{3}}{2}x  \bigg) \bigg] \rightarrow \mathcal{\overline{O}}_{x}=0.$$

The possible solutions of \eqref{eq:Equation_50} are similar to those of \ref{ex:Example_6} with the addition of a particular solution associated with $ \partial_{x}\sinh x \partial_{x}Q(u)=g(x,t)$, which satisfies $Q(u)= \sum_{j=0}^{n_3}C_j t^{j}e^{-\frac{x}{2}}cos\bigg(\frac{2\pi}{3}+\frac{\sqrt{3}}{2}x  \bigg)$ i.e.

\begin{equation}\label{eq:Equation_51}	\alpha_{k'}(0,-1,0,1,h_1(t)\ln\bigg(\tanh \bigg(\frac{x}{2}\bigg)\bigg)+h_2(t) )=\zeta_{3}^{k'}v_1(x,t)+\zeta_{3}^{3-k'}v_2(x,t),	
\end{equation}
with $ 1 \leq k' \leq 3$ and

$$v_1(x,t)=\sqrt[3]{\frac{h_1(t)\ln\big(\tanh \big(\frac{x}{2}\big)\big)+h_2(t)+\overline{v}(x,t)}{2}+\sqrt{\frac{(h_1(t)\ln\big(\tanh \big(\frac{x}{2}\big)\big)+h_2(t)+\overline{v}(x,t))^{2}}{4}-\frac{1}{27}}}$$

$$ v_2 (x,t)=\sqrt[3]{\frac{h_1(t)\ln\big(\tanh \big(\frac{x}{2}\big)\big)+h_2(t)+\overline{v}(x,t)}{2}-\sqrt{\frac{(h_1(t)\ln\big(\tanh \big(\frac{x}{2}\big)\big)+h_2(t)+\overline{v}(x,t))^{2}}{4}-\frac{1}{27}}}	,$$

with $h_1(t)$, $h_2(t)$ defined in \eqref{eq:Equation_48}, and $\overline{v}(x,t)$ is a particular solution of $ \partial_{x}\sinh x \partial_{x}v=g(x,t)$ with $g(x,t)=\sum_{j=0}^{n_3}C_j t^{j}e^{-\frac{x}{2}}\bigg[\frac{1}{\tanh x} \cos\bigg(\frac{4\pi}{3}+\frac{\sqrt{3}}{2}x  \bigg)+\cos \bigg(\frac{\sqrt{3}}{2}x  \bigg) \bigg] $ i.e.

$$ \overline{v}(x,t)=\sum_{j=0}^{n_3}C_j t^{j}e^{-\frac{x}{2}}\cos\bigg(\frac{2\pi}{3}+\frac{\sqrt{3}}{2}x  \bigg).$$ 

\end{example}

\begin{example}\label{ex:Example_8}
	In this example, we will use a PDE non-linear equation by making use of invariant spaces, i.e.
	
	\begin{equation}\label{eq:Equation_52}
		\partial_{t}^{\nu}u=\bigg(1-\frac{1}{\tanh^{2}x}\bigg)u^{2}\partial_{x} u+ \frac{1}{\tanh x}u^{2}\partial_{x}^{(2)}u+u^{2}\partial_{x}^{(3)}u+\bigg(1-\frac{1}{\tanh^{2}x}\bigg)u\partial_{x} u+ \frac{1}{\tanh x}u\partial_{x}^{(2)}u+u\partial_{x}^{(3)}u+ 
	\end{equation}
$$\frac{1}{\tanh x}\partial_{x} u+\partial_{x}^{(2)}u+ \lambda u,$$

with $\lambda \in \mathbb{C}$, and we recall that $\partial_{t}^{\nu}$ fractional derivative of Caputo with $\nu \in (0,1)$. In this case we can rewrite the second member $\mathcal{O}_{x}u=\mathcal{\overline{O}}_{x}u + \lambda u$. So we have to find the space of certain functions such that $\mathcal{\overline{O}}_{x}u=0$, so, as we have done in the previous examples, we find the elementary solutions by nested derivation, i.e., we verify if $\mathcal{\overline{O}}_{x}u$ is an algebraic combination of opportune nested derivations.
Let's identify the factors that members have in common in \eqref{eq:Equation_52}, and we obtain

$$\partial_{t}^{\nu}u= \bigg(1-\frac{1}{\tanh^{2}x}\bigg)(u^{2}+u)\partial_{x} u+\frac{1}{\tanh x}(u^{2}+u)\partial_{x}^{(2)}u+(u^{2}+u)\partial_{x}^{(3)}u+\frac{1}{\tanh x}\partial_{x} u+\partial_{x}^{(2)}u+ \lambda u$$

$$\partial_{t}^{\nu}u= (u^{2}+u)\bigg [\bigg(1-\frac{1}{\tanh^{2}x}\bigg)\partial_{x}u+\frac{1}{\tanh x}\partial_{x}^{(2)}u+  \partial_{x}^{(3)}u \bigg]+ \frac{1}{\sinh x}\partial_{x} \sinh x \partial_{x}u + \lambda u$$

$$\partial_{t}^{\nu}u= (u^{2}+u)\partial_{x} \frac{1}{\sinh x}\partial_{x} \sinh x \partial_{x}u +\frac{1}{\sinh x}\partial_{x} \sinh x \partial_{x}u + \lambda u.$$

\quad
Hence $\mathcal{\overline{O}}_{x}u=0$ if $\partial_{x} \sinh x \partial_{x}u =0$; this equation yields a vectorial space of possible solutions $< \ln\big(\tanh \big( \frac{x}{2}\big)\big),1>$, i.e., $h_1(t) \ln\big(\tanh \big( \frac{x}{2}\big)\big)+h_2(t)$, where $h_1(t)$, $h_2(t)$ are chosen such that $\partial_{t}^{\nu}h_1(t)=\lambda  h_1(t)$ and $\partial_{t}^{\nu}h_2(t)=\lambda  h_2(t)$.
Therefore the possible solutions of \eqref{eq:Equation_52} are

\begin{equation}\label{eq:Equation_53}
	u(x,t)=E_{\nu}(\lambda t)\bigg( A\ln\bigg(\tanh \bigg( \frac{x}{2}\bigg)\bigg) +B  \bigg),
\end{equation}

where $A,B \in \mathbb{C}$ and $E_{\nu}(\lambda t)=\sum_{k=0}^{\infty}\frac{(\lambda t)^{k}}{\Gamma(\nu k +1)}$ is one-parameter-Mittag-Leffler function.

\end{example}

\begin{example}\label{ex:Example_9}
	In this example the equation that we want to illustrate is \eqref{eq:Equation_52} by adding the following terms
	
	\begin{equation}\label{eq:Equation_54}
		\partial_{t}^{\nu}u=\bigg(1-\frac{1}{\tanh^{2}x}\bigg)u^{2}\partial_{x} u+ \frac{1}{\tanh x}u^{2}\partial_{x}^{(2)}u+u^{2}\partial_{x}^{(3)}u+\bigg(1-\frac{1}{\tanh^{2}x}\bigg)u\partial_{x} u+ \frac{1}{\tanh x}u\partial_{x}^{(2)}u+u\partial_{x}^{(3)}u+ 
		\end{equation}
		$$\frac{1}{\tanh x}\partial_{x} u+\partial_{x}^{(2)}u+ \lambda u 				
		-E_{\nu}(\lambda t)e^{-\frac{x}{2}}\bigg\{\cos \bigg(\frac{2\pi}{3}+\frac{\sqrt{3}}{2}x  \bigg)+ \frac{1}{\tanh x}\bigg[\cos \bigg(\frac{\sqrt{3}}{2}x  \bigg)+  \frac{\tanh^{2}x-1}{\tanh x}\cos\bigg(\frac{4\pi}{3}+\frac{\sqrt{3}}{2}x  \bigg)\bigg]\bigg \}u^{2}- 
$$

$$ E_{\nu}(\lambda t)e^{-\frac{x}{2}}\bigg\{\cos \bigg(\frac{2\pi}{3}+\frac{\sqrt{3}}{2}x  \bigg)+ \frac{1}{\tanh x}\bigg[\cos \bigg(\frac{\sqrt{3}}{2}x  \bigg)+  \frac{\tanh^{2}x-1}{\tanh x}\cos\bigg(\frac{4\pi}{3}+\frac{\sqrt{3}}{2}x  \bigg)\bigg]\bigg \}u-$$
$$ E_{\nu}(\lambda t) e^{-\frac{x}{2}}\bigg[\frac{1}{\tanh x} \cos\bigg(\frac{4\pi}{3}+\frac{\sqrt{3}}{2}x  \bigg)+\cos \bigg(\frac{\sqrt{3}}{2}x  \bigg) \bigg]. $$  

The second member of \eqref{eq:Equation_54}, as in the previous example \ref{ex:Example_8}, can rewrite $\mathcal{O}_{x}u=\mathcal{\overline{O}}_{x}u + \lambda u$, and we can find elementary solutions by the invariant space method resolving $\mathcal{\overline{O}}_{x}u=0$ and since it turns out that the equation \eqref{eq:Equation_54} becomes

$$ 	\partial_{t}^{\nu}u=(u^{2}+u)\bigg[\partial_{x} \frac{1}{\sinh x}\partial_{x} \sinh x \partial_{x}u -$$

$$  E_{\nu}(\lambda t)e^{-\frac{x}{2}}\bigg\{\cos \bigg(\frac{2\pi}{3}+\frac{\sqrt{3}}{2}x  \bigg)+ \frac{1}{\tanh x}\bigg[\cos \bigg(\frac{\sqrt{3}}{2}x  \bigg)+  \frac{\tanh^{2}x-1}{\tanh x}\cos\bigg(\frac{4\pi}{3}+\frac{\sqrt{3}}{2}x  \bigg)\bigg]\bigg \}                     \bigg] +$$

$$\bigg\{\frac{1}{\sinh x}\partial_{x} \sinh x \partial_{x}u-E_{\nu}(\lambda t) e^{-\frac{x}{2}}\bigg[\frac{1}{\tanh x} \cos\bigg(\frac{4\pi}{3}+\frac{\sqrt{3}}{2}x  \bigg)+\cos \bigg(\frac{\sqrt{3}}{2}x  \bigg) \bigg]\bigg\}+ \lambda u.$$

\quad

Therefore

$$\frac{1}{\sinh x}\partial_{x} \sinh x \partial_{x}u=E_{\nu}(\lambda t) e^{-\frac{x}{2}}\bigg[\frac{1}{\tanh x} \cos\bigg(\frac{4\pi}{3}+\frac{\sqrt{3}}{2}x  \bigg)+\cos \bigg(\frac{\sqrt{3}}{2}x  \bigg) \bigg] \rightarrow \mathcal{\overline{O}}_{x}u=0.$$

We thus obtain a particular solution of the $\frac{1}{\sinh x}\partial_{x} \sinh x \partial_{x}u=E_{\nu}(\lambda t) e^{-\frac{x}{2}}\bigg[\frac{1}{\tanh x} \cos\bigg(\frac{4\pi}{3}+\frac{\sqrt{3}}{2}x  \bigg)+\cos \bigg(\frac{\sqrt{3}}{2}x  \bigg) \bigg]$ and of the equation \eqref{eq:Equation_54} which is calculated in similar way at example \ref{ex:Example_7} respect with $\frac{1}{\sinh x}\partial_{x} \sinh x \partial_{x}v=P_n(t)e^{-\frac{x}{2}}\bigg[\frac{1}{\tanh x} \cos\bigg(\frac{4\pi}{3}+\frac{\sqrt{3}}{2}x  \bigg)+\cos \bigg(\frac{\sqrt{3}}{2}x  \bigg) \bigg]$ i.e.

$$\overline{u}(x,t)=E_{\nu}(\lambda t)e^{-\frac{x}{2}}\cos\bigg(\frac{2\pi}{3}+\frac{\sqrt{3}}{2}x  \bigg),$$

combining this with the solutions in the previous example, we obtain the following families of possibily elementary solutions of equation \eqref{eq:Equation_54}

\begin{equation}\label{eq:Equation_55}
u(x,t)=E_{\nu}(\lambda t)\bigg( A\ln\bigg(\tanh \bigg( \frac{x}{2}\bigg)\bigg) +B + CE_{\nu}(\lambda t)e^{-\frac{x}{2}}\cos\bigg(\frac{2\pi}{3}+\frac{\sqrt{3}}{2}x  \bigg) \bigg)	
\end{equation}
where $A,B,C \in \mathbb{C}$.   
\end{example}

\begin{example}\label{ex:Example_10}
	\quad
	
In this example we will discuss the equation similar to examples \ref{pr:Proposition_8} and \ref{ex:Example_9} replacing the nested derivative non-linear with a linear one, and this we will see subsequently.

However, given the following equation

\begin{equation}\label{eq:Equation_56}
	\partial_{t}^{\nu}u=\bigg(1-\frac{1}{\tanh^{2}x}\bigg)nu^{n+1}\partial_{x}u+ \frac{1}{\tanh x}nu^{n+1}\partial_{x}^{(2)}u+\frac{1}{\tanh x}n(n-1)u^{n}(\partial_{x}u)^{2}+nu^{n+1}\partial_{x}^{(3)}u+3n(n-1)u^{n}\partial_{x}u\partial_{x}^{(2)}u+	
\end{equation} 	

$$n(n-1)(n-2)u^{n-1}(\partial_{x}u)^{3}+\bigg(1-\frac{1}{\tanh^{2}x}\bigg)nu^{n}\partial_{x}u+\frac{1}{\tanh x}nu^{n}\partial_{x}^{2}u+n(n-1)\frac{1}{\tanh x}u^{n-1}(\partial_{x}u)^{2}+nu^{n}\partial_{x}^{(3)}u+$$
$$3n(n-1)u^{n-1}\partial_{x}u\partial_{x}^{(2)}u+n(n-1)(n-2)u^{n-2}(\partial_{x}u)^{3}+\frac{1}{\tanh x}nu^{n-1}\partial_{x}u+ nu^{n-1}\partial_{x}^{(2)}u+n(n-1)u^{n-2}(\partial_{x}u)^{2}+ \lambda u,$$

with with $\lambda \in \mathbb{C}$ and we recall that $\partial_{t}^{\nu}$ fractional derivative of Caputo with $\nu \in (0,1)$.

Applying Faa' di Bruno's formula for $u^{n}$ and the coefficients of the nested derivative formula \eqref{eq:Equation_9}, and \eqref{eqs:Equations_2} to the second member of \eqref{eq:Equation_56} $\mathcal{O}_{x}$, we can see that the nested derivative, which is an algebraic combination of the second member of $\mathcal{O}_{x}$, is grade $3$ and the functions that defined the nested derivative are $f_1(x)=\sinh x$, $f_2(x)=\frac{1}{\sinh x}$. Furthermore, $\frac{1}{f_1(x)}$ is an antiderivative elementary function.

To summarize, we have that

$$\partial_{t}^{\nu}u=(u^{2}+u)\partial_{x} \frac{1}{\sinh x}\partial_{x} \sinh x \partial_{x}u^{n}+\frac{1}{\sinh x}\partial_{x} \sinh x \partial_{x}u^{n}+\lambda u.$$

Hence

$$\mathcal{O}_{x}=\lambda u \hspace{0.1cm}if\hspace{0.1cm}(u^{2}+u)\partial_{x} \frac{1}{\sinh x}\partial_{x} \sinh x \partial_{x}u^{n}+\frac{1}{\sinh x}\partial_{x} \sinh x \partial_{x}u^{n}=0$$

and 

$$(u^{2}+u)\partial_{x} \frac{1}{\sinh x}\partial_{x} \sinh x \partial_{x}u^{n}+\frac{1}{\sinh x}\partial_{x} \sinh x \partial_{x}u^{n}=0\hspace{0.1cm}if\hspace{0.1cm}\frac{1}{\sinh x}\partial_{x} \sinh x \partial_{x}u^{n}=0.$$

We have that

$$\frac{1}{\sinh x}\partial_{x} \sinh x \partial_{x}u^{n}=0\hspace{0.1cm}if\hspace{0.1cm} u(x,t)=f(t)\zeta_{n}^{j}\sqrt[n]{A\ln\bigg(\tanh \bigg( \frac{x}{2}\bigg)\bigg)+B}=f(t)\sqrt[n]{A_1 \ln\bigg(\tanh \bigg( \frac{x}{2}\bigg)\bigg)+B_2},$$

where $\zeta_{n}^{j}$ is a primitive root of order $n$ with $1\leq j \leq n$ and $A_1=\zeta_{n}^{jn}A$, $B_2=\zeta_{n}^{jn}B$ with $A,B \in \mathbb{C}$.

\quad

We have thus elementary solutions; in fact, $u^{n}-d=0$ is resolved by radicals, and it can apply the proposition \ref{pr:Proposition_9}.

To determine $f(t)$, we must calculate 

$$\partial_{t}^{\nu}f(t)=\lambda f(t).$$

As we have seen in the previous examples, $f(t)=E_{\nu}(\lambda t)$, and therefore the possible solutions of \eqref{eq:Equation_56} are

\begin{equation}\label{eq:Equation_57}
u(x,t)=E_{\nu}(\lambda t)\sqrt[n]{A_1 \ln\bigg(\tanh \bigg( \frac{x}{2}\bigg)\bigg)+B_2}.	
\end{equation}

\end{example}

\begin{example}\label{ex:Example_11}
As in the examples \ref{ex:Example_8} and \ref{ex:Example_9}, we consider a modified version of equation \eqref{eq:Equation_56} by adding the term used in the example \ref{ex:Example_9}, and we obtain

\begin{equation}\label{eq:Equation_58}
		\partial_{t}^{\nu}u=\bigg(1-\frac{1}{\tanh^{2}x}\bigg)nu^{n+1}\partial_{x}u+ \frac{1}{\tanh x}nu^{n+1}\partial_{x}^{(2)}u+\frac{1}{\tanh x}n(n-1)u^{n}(\partial_{x}u)^{2}+nu^{n+1}\partial_{x}^{(3)}u+3n(n-1)u^{n}\partial_{x}u\partial_{x}^{(2)}u+	
		\end{equation} 	
		
		$$n(n-1)(n-2)u^{n-1}(\partial_{x}u)^{3}+\bigg(1-\frac{1}{\tanh^{2}x}\bigg)nu^{n}\partial_{x}u+\frac{1}{\tanh x}nu^{n}\partial_{x}^{2}u+n(n-1)\frac{1}{\tanh x}u^{n-1}(\partial_{x}u)^{2}+nu^{n}\partial_{x}^{(3)}u+$$
		$$3n(n-1)u^{n-1}\partial_{x}u\partial_{x}^{(2)}u+n(n-1)(n-2)u^{n-2}(\partial_{x}u)^{3}+\frac{1}{\tanh x}nu^{n-1}\partial_{x}u+ nu^{n-1}\partial_{x}^{(2)}u+n(n-1)u^{n-2}(\partial_{x}u)^{2}+ \lambda u-$$

$$E_{\nu}(\lambda t)e^{-\frac{x}{2}}\bigg\{\cos \bigg(\frac{2\pi}{3}+\frac{\sqrt{3}}{2}x  \bigg)+ \frac{1}{\tanh x}\bigg[\cos \bigg(\frac{\sqrt{3}}{2}x  \bigg)+  \frac{\tanh^{2}x-1}{\tanh x}\cos\bigg(\frac{4\pi}{3}+\frac{\sqrt{3}}{2}x  \bigg)\bigg]\bigg \}u^{2}- 
$$

$$ E_{\nu}(\lambda t)e^{-\frac{x}{2}}\bigg\{\cos \bigg(\frac{2\pi}{3}+\frac{\sqrt{3}}{2}x  \bigg)+ \frac{1}{\tanh x}\bigg[\cos \bigg(\frac{\sqrt{3}}{2}x  \bigg)+  \frac{\tanh^{2}x-1}{\tanh x}\cos\bigg(\frac{4\pi}{3}+\frac{\sqrt{3}}{2}x  \bigg)\bigg]\bigg \}u-$$
$$ E_{\nu}(\lambda t) e^{-\frac{x}{2}}\bigg[\frac{1}{\tanh x} \cos\bigg(\frac{4\pi}{3}+\frac{\sqrt{3}}{2}x  \bigg)+\cos \bigg(\frac{\sqrt{3}}{2}x  \bigg) \bigg]. $$  

Using the procedure of the example \ref{ex:Example_9}, the possible solutions of \eqref{eq:Equation_58} are the eigenfunctions of the second member of \eqref{eq:Equation_58} that belong to the kernel of the operator

$$\frac{1}{\sinh x}\partial_{x} \sinh x \partial_{x}u^{n}-E_{\nu}(\lambda t) e^{-\frac{x}{2}}\bigg[\frac{1}{\tanh x} \cos\bigg(\frac{4\pi}{3}+\frac{\sqrt{3}}{2}x  \bigg)+\cos \bigg(\frac{\sqrt{3}}{2}x  \bigg) \bigg].$$

Without taking into account any initial conditions in space-time, the solutions are
\begin{equation}\label{eq:Equation_59}
u(x,t)=E_{\nu}(\lambda t)\sqrt[n]{A \ln\bigg(\tanh \bigg( \frac{x}{2}\bigg)\bigg)+Be^{-\frac{x}{2}}\cos\bigg(\frac{2\pi}{3}+\frac{\sqrt{3}}{2}x  \bigg)+C},
\end{equation}
with $A,B,C \in \mathbb{C}$.
	
\end{example}

\section{Appendix: A generalization of the $\sin(x)$ and $\cos(x)$ functions}\label{se:Section_7}

Recalling that the real function $\cos(x)$ can be rewritten in terms of complex numbers and opportune complex functions in this way

$$\cos(x)=\frac{e^{i x}+e^{-i x}}{2}.$$

We can generalize by replacing $i$ with a primitive root of unity with opportune raising power $\zeta_{n}^{k}$ with its conjugate $\overline{\zeta_{n}^{k}}=\zeta_{n}^{n-k}$ and we obtain

$$f(x,\zeta_{n}^{k}):=\frac{e^{\zeta_{n}^{k} x}+e^{\overline{\zeta_{n}^{k}} x}}{2}$$

Let's envelope $f(x,\zeta_{n}^{k})$ in terms of exponential 

$$\frac{e^{\zeta_{n}^{k} x}+e^{\overline{\zeta_{n}^{k}} x}}{2}=\frac{e^{e^{\frac{2k\pi i}{n}} x}+e^{e^{-\frac{2k\pi i}{n} x}}}{2}=\frac{e^{\big[\cos\big(\frac{2k\pi }{n}\big)+i\sin\big(\frac{2k\pi }{n}\big)\big] x}+e^{\big[\cos\big(\frac{2k\pi }{n}\big)-i\sin\big(\frac{2k\pi }{n}\big)\big] x}}{2}= $$

$$=e^{x\cos\big(\frac{2k\pi }{n}\big)}\frac{e^{x i\sin\big(\frac{2k\pi }{n}\big) }+e^{-ix\sin\big(\frac{2k\pi }{n}\big) }}{2}=e^{x\cos\big(\frac{2k\pi }{n}\big)}\cos\bigg(x\sin\bigg(\frac{2k\pi }{n}\bigg)\bigg).$$

So

\begin{equation}\label{eq:Equation_60}
f(x,\zeta_{n}^{k})=e^{x\cos\big(\frac{2k\pi }{n}\big)}\cos\bigg(x\sin\bigg(\frac{2k\pi }{n}\bigg)\bigg).	
\end{equation}
Furthermore, $f(x)$ is a real function, and $f(x, \zeta_{4})=\cos(x)$ is the particular case for $k=1$ and $n=4$.

We want to prove that 

\begin{equation}\label{eq:Equation_61}
	\bigg(e^{x\cos\big(\frac{2k\pi }{n}\big)}\cos\bigg(x\sin\bigg(\frac{2k\pi }{n}\bigg)\bigg) \bigg)^{(j)}=e^{x\cos\big(\frac{2k\pi }{n}\big)}\cos\bigg(\frac{2jk\pi }{n}+ x\sin\bigg(\frac{2k\pi }{n}\bigg)\bigg).
\end{equation}

Let's prove this by induction on the grade of derivation.

We begin with $j=1$.

$$\bigg(e^{x\cos\big(\frac{2k\pi }{n}\big)}\cos\bigg(x\sin\bigg(\frac{2k\pi }{n}\bigg)\bigg) \bigg)^{'}=\bigg(\frac{e^{e^{\frac{2k\pi i}{n}} x}+e^{e^{-\frac{2k\pi i}{n} x}}}{2}\bigg)^{'}=\frac{e^{\frac{2k\pi i}{n}}e^{e^{\frac{2k\pi i}{n}} x}+e^{-\frac{2k\pi i}{n}}e^{e^{-\frac{2k\pi i}{n} x}}}{2}=$$

$$=\frac{e^{\frac{2k\pi i}{n}}e^{\big[\cos\big(\frac{2k\pi }{n}\big)+i\sin\big(\frac{2k\pi }{n}\big)\big] x}+e^{-\frac{2k\pi i}{n}}e^{\big[\cos\big(\frac{2k\pi }{n}\big)-i\sin\big(\frac{2k\pi }{n}\big)\big] x}}{2}=e^{x\cos\big(\frac{2k\pi }{n}\big)}\frac{e^{ i\big(\frac{2k\pi }{n}+x\sin\big(\frac{2k\pi }{n}\big) \big)}+e^{-i\big(\frac{2k\pi }{n}+x\sin\big(\frac{2k\pi }{n}\big)\big) }}{2}=$$

$$e^{x\cos\big(\frac{2k\pi }{n}\big)}\cos\bigg(\frac{2k\pi }{n}+ x\sin\bigg(\frac{2k\pi }{n}\bigg)\bigg).$$
So the base induction for $j=1$ is proved. Now let's move on to the inductive hypothesis for $j\geq 2$ i.e.

$$ \bigg(e^{x\cos\big(\frac{2k\pi }{n}\big)}\cos\bigg(x\sin\bigg(\frac{2k\pi }{n}\bigg)\bigg) \bigg)^{(j-1)}=e^{x\cos\big(\frac{2k\pi }{n}\big)}\cos\bigg(\frac{2(j-1)k\pi }{n}+ x\sin\bigg(\frac{2k\pi }{n}\bigg)\bigg).   $$

$$\bigg(e^{x\cos\big(\frac{2k\pi }{n}\big)}\cos\bigg(x\sin\bigg(\frac{2k\pi }{n}\bigg)\bigg) \bigg)^{(j)}=\bigg(\bigg(e^{x\cos\big(\frac{2k\pi }{n}\big)}\cos\bigg(x\sin\bigg(\frac{2k\pi }{n}\bigg)\bigg) \bigg)^{(j-1)} \bigg)^{'}=$$
$$=\bigg(e^{x\cos\big(\frac{2k\pi }{n}\big)}\frac{e^{ i\big(\frac{2(j-1)k\pi }{n}+x\sin\big(\frac{2k\pi }{n}\big) \big)}+e^{-i\big(\frac{2(j-1)k\pi }{n}+x\sin\big(\frac{2k\pi }{n}\big)\big) }}{2}\bigg)^{'}=$$
$$=\bigg(\frac{e^{\frac{2(j-1)k\pi i}{n}}e^{\big[\cos\big(\frac{2k\pi }{n}\big)+i\sin\big(\frac{2k\pi }{n}\big)\big] x}+e^{-\frac{2(j-1)k\pi i}{n}}e^{\big[\cos\big(\frac{2k\pi }{n}\big)-i\sin\big(\frac{2k\pi }{n}\big)\big] x}}{2}\bigg)^{'}=$$

$$=\frac{e^{\frac{2(j-1)k\pi i}{n}}e^{\frac{2k\pi i}{n}}e^{\big[\cos\big(\frac{2k\pi }{n}\big)+i\sin\big(\frac{2k\pi }{n}\big)\big] x}+e^{-\frac{2(j-1)k\pi i}{n}}e^{-\frac{2k\pi i}{n}}e^{\big[\cos\big(\frac{2k\pi }{n}\big)-i\sin\big(\frac{2k\pi }{n}\big)\big] x}}{2}=$$

$$=\frac{e^{\frac{2jk\pi i}{n}}e^{\big[\cos\big(\frac{2k\pi }{n}\big)+i\sin\big(\frac{2k\pi }{n}\big)\big] x}+e^{-\frac{2jk\pi i}{n}}e^{\big[\cos\big(\frac{2k\pi }{n}\big)-i\sin\big(\frac{2k\pi }{n}\big)\big] x}}{2} =$$
$$=e^{x\cos\big(\frac{2k\pi }{n}\big)}\frac{e^{ i\big(\frac{2jk\pi }{n}+x\sin\big(\frac{2k\pi }{n}\big) \big)}+e^{-i\big(\frac{2jk\pi }{n}+x\sin\big(\frac{2k\pi }{n}\big)\big) }}{2}=e^{x\cos\big(\frac{2k\pi }{n}\big)}\cos\bigg(\frac{2jk\pi }{n}+ x\sin\bigg(\frac{2k\pi }{n}\bigg)\bigg).                       $$

In fact, $\cos(x)=e^{x\cos\big(\frac{2\pi }{4}\big)}\cos\bigg( x\sin\bigg(\frac{2\pi }{4}\bigg)\bigg)=e^{x\cos\big(\frac{\pi }{2}\big)}\cos\bigg( x\sin\bigg(\frac{\pi }{2}\bigg)\bigg)=e^{x \cdot 0}\cos\bigg( x\cdot 1\bigg)$ so $(\cos(x))^{'}=\cos\big(\frac{2 \pi}{4}+x\big)=\cos\big(\frac{\pi}{2}+x\big)=-\sin(x)$.

To conclude this section, it can also be deduced from \eqref{eq:Equation_61} that for every $k \in \mathbb{Z}$ we have

$$\bigg(e^{x\cos\big(\frac{2k\pi }{n}\big)}\cos\bigg(x\sin\bigg(\frac{2k\pi }{n}\bigg)\bigg) \bigg)^{(n)}=e^{x\cos\big(\frac{2k\pi }{n}\big)}\cos\bigg(x\sin\bigg(\frac{2k\pi }{n}\bigg)\bigg).$$

\end{document}